\documentclass[11pt,letterpaper]{article}
\usepackage{amsthm,amsmath,amssymb,amsfonts} 
\usepackage{epsfig} 
\usepackage{latexsym,nicefrac,bbm}
\usepackage{xspace}
\usepackage{color,fancybox,graphicx,subcaption,fullpage}
\usepackage[top=1.25in, bottom=1.25in, left=1.25in, right=1.25in]{geometry}
\usepackage{tabularx}
\usepackage{hyperref} 
\usepackage{pdfsync}
\usepackage{multicol}
\usepackage{cite}
\usepackage[capitalize]{cleveref}

\renewcommand{\epsilon}{\varepsilon}

\usepackage{epsfig}
\usepackage{verbatim}

\theoremstyle{definition}

\usepackage{algorithmicx}
\usepackage{algorithm,caption}
\usepackage{algpseudocode}

\usepackage{mhequ}
\def \be{\begin{equs}}
\def \ee{\end{equs}}

\newtheorem{theorem}{Theorem}[section]

\newtheorem{definition}[theorem]{Definition}

\newtheorem{proposition}[theorem]{Proposition}

\newtheorem*{theorem*}{Theorem}
\newtheorem{remark}[theorem]{Remark}
\newtheorem{example}[theorem]{Example}

\crefname{theorem}{Theorem}{Theorems}
\crefname{observation}{Observation}{Observations}
\crefname{proposition}{Proposition}{Propositions}
\crefname{claim}{Claim}{Claims}
\crefname{condition}{Condition}{Conditions}
\crefname{example}{Example}{Examples}
\crefname{fact}{Fact}{Facts}
\crefname{lemma}{Lemma}{Lemmas}
\crefname{corollary}{Corollary}{Corollaries}
\crefname{definition}{Definition}{Definitions}
\crefname{remark}{Remark}{Remarks}
\crefname{example}{Example}{Examples}
\crefname{exercise}{Exercise}{Exercises}

\title{\bf Geodesic Convex Optimization: \\ Differentiation on Manifolds, Geodesics, and Convexity}

\author{Nisheeth K. Vishnoi\footnote{Many thanks to Ozan Y{\i}ld{\i}z for preparing the scribe notes of my lecture on this topic and making the figures.} \\ EPFL, Switzerland} 

\usepackage{mathtools}
\newcommand{\R}{\mathbb{R}}
\newcommand{\Z}{\mathbb{Z}}

\DeclarePairedDelimiterX{\Set}[2]\{\}{%
  \, #1 \;\delimsize\vert\; #2 \,
}
\DeclarePairedDelimiter{\abs}{\lvert}{\rvert}
\DeclarePairedDelimiter{\norm}{\lVert}{\rVert}
\DeclarePairedDelimiter{\inner}{\langle}{\rangle}
\newcommand{\inparen}[1]{\left(#1\right)}             

\newcommand{\Sn}{\mathbb{S}}

\newcommand{\diag}{\mathrm{diag}}
\newcommand{\tr}{\mathrm{tr}}

\newcommand{\intersection}{\cap}

\begin{document}

\maketitle
  
  \thispagestyle{empty}

\begin{abstract}
Convex optimization is a vibrant and successful area due to the existence of a variety of efficient algorithms that leverage the rich structure provided by convexity.
Convexity of a smooth set or a function in a Euclidean space is defined by how it interacts with the standard differential structure in this space -- the Hessian of a convex function has to be positive semi-definite everywhere. 
However, in recent years, there is a growing demand to understand non-convexity and develop computational methods to optimize non-convex functions.
Intriguingly, there is a type of non-convexity that disappears once one introduces a suitable differentiable structure and redefines convexity with respect to the straight lines, or {\em geodesics}, with respect to this structure. 
Such convexity is referred to as {\em geodesic convexity}. 
Interest in studying it arises due to recent reformulations of some non-convex  problems as geodesically convex optimization problems over geodesically convex sets. 
Geodesics on manifolds have been extensively studied in various branches of Mathematics and Physics.
However, unlike convex optimization,  understanding geodesics and  geodesic convexity from a computational point of view largely remains a mystery. 
The goal of this exposition is to introduce the first part of geodesic convex optimization -- geodesic convexity -- in a self-contained manner.
We first present a variety of notions from differential and Riemannian geometry such as differentiation on manifolds, geodesics, and then introduce geodesic convexity.
We conclude by showing that certain non-convex optimization problems such as  computing the Brascamp-Lieb constant and the operator scaling problem have geodesically convex formulations.
\end{abstract}

\newpage

\tableofcontents

\newpage

\section{Beyond Convexity}
In the most general setting, an optimization problem takes the form
\begin{equ}[eq:general-optimization]
\inf_{x\in K} f(x),
\end{equ}
for some set $K$ and some function $f:K\to\R$.
In the  case when $K \subseteq \R^d$, we can talk about the convexity of $K$ and $f$. 
$K$ is said to be convex if any ``straight line'' joining two points in $K$ is entirely contained in $K$, and $f$ is said to be convex if on any such straight line, the average value of $f$ at the end points is at least the value of $f$ at the mid-point of the line.
When $f$ is ``smooth'' enough, there are equivalent definitions of convexity in terms of the standard differential structure in $\R^d$: the gradient or the Hessian of $f$.
Thus, convexity can also be viewed as a property arising from the interaction of the function and how we differentiate in $\R^n$; e.g., the Hessian of $f$ at every point in $K$ should be positive semi-definite. 
When both $K$ and $f$ are convex, the optimization problem in \eqref{eq:general-optimization} is called a convex optimization problem.
The fact that the convexity of $f$ implies that any local minimum of $f$ in $K$ is also a global minimum, along with the fact that computing gradients and Hessians is typically easy in Euclidean spaces, makes it well-suited for developing  algorithms such as gradient descent, interior point methods, and cutting plane methods.
Analyzing the convergence of these methods boils down to understanding how well-behaved  derivatives of the function are; see \cite{boyd2004convex,nesterov2004introductory,vishnoi_2018} for more on  algorithms for convex optimization.

 \paragraph{Geodesic convexity.} In recent times, several non-convex optimization problems have emerged and, as a consequence, there is a need to understand non-convexity and develop methods for such problems. 
Interestingly, there is a type of non-convexity that disappears when we view the domain as a manifold and redefine what we mean by a straight line on it.
This redefinition of a straight line entails the introduction of a differential structure on $\R^n$, or, more generally, on a ``smooth manifold''. 
Roughly speaking, a manifold is a topological space that locally looks like a Euclidean space.
''Differentiable manifolds'' are a special class of manifolds that come with a differential structure that allows one to do calculus over them.
Straight lines on differential manifolds are called ``geodesics'', and  a set that has the property that a geodesic joining any two points in it is entirely contained in the set is called geodesically convex (with respect to the given differential structure).
A function that has this property that  its average value at the end points of a geodesic is at least the value of $f$ at the mid-point of the geodesic is called geodesically convex (with respect to the given differential structure). 
And, when $K$ and $f$ are both geodesically convex, the optimization problem in \eqref{eq:general-optimization} is called a geodesically convex optimization problem.
Geodesically convex functions also have key properties similar to convex functions such as the fact that a local minimum is also a global minimum.

\paragraph{Geodesics.} Geodesics on manifolds have been well-studied, in various branches of Mathematics and Physics. 
Perhaps the most famous use of them is in the Theory of General Relativity by Einstein \cite{einstein1916foundation}.
In general, there is no unique notion of a geodesic on a smooth manifold -- it depends on the differential structure. 
However, if one imposes an additional  ``metric'' structure on the manifold that allows us to measure lengths and angles, there is an alternative, and equivalent, way to define geodesics -- as shortest paths with respect to this metric.
The most famous class of such metrics give rise to  Riemannian manifolds -- these are manifolds where each point has an associated local inner product matrix that is positive semi-definite. 
The fundamental theorem of Riemannian geometry states  that any differential structure that is ``compatible'' with a Riemannian metric is unique and thus we can view geodesics as either straight lines or distance minimizing curves.\footnote{This even holds for what are called pseudo-Riemannian manifolds that arise in relativity.} 

\paragraph{Applications.} Geodesic convex optimization has recently become interesting in computer science and machine learning due to the realization that some important problems that appear to be non-convex at first glance, are geodesically convex if we introduce a suitable differential structure and a metric.
Computationally, however, there is an additional burden to be able to compute geodesics and argue about quantities such as areas and volumes that may no longer have closed form solutions.
For instance, consider the problem  
\begin{equation*}
\inf_{x >0} \log p(x) - \sum_i \log x_i
\end{equation*}
for a polynomial $p(x) \in \R_+[x_1,\ldots, x_n]$.
This problem and its variants arise in computing maximum entropy distributions; see e.g., \cite{gurvits2006hyperbolic, SV17RealStable}.
While this problem is clearly non-convex, it turns out that it is geodesically convex if we consider the positive orthant $\R_+^n$ endowed with the usual differential structure but the Riemannian metric arising as the Hessian of the function 
$$-\sum_i \log x_i.$$
This viewpoint, while not new, can be used to rewrite the optimization problem above as a convex optimization problem in new variables (by replacing $x_i$ by $e^{y_i}$ for $y_i \in \mathbb{R}$)
  and use algorithms from convex optimization to solve it efficiently; see \cite{singh2014entropy, straszak2017entropy}. 

An extension of the above optimization problem to matrices is the Operator Scaling problem. Given an operator $T$ that maps positive definite matrices to positive definite matrices, compute 
\begin{equation*}
\inf_{X \succ 0} \log \det T(X) - \log \det X.
\end{equation*}
This problem was introduced in \cite{gurvits2004classical} and was studied  in the work of \cite{GargGOW16,garg2016algorithmic}. 
This problem is non-convex, but unlike the previous example, there is no obvious way to convexity the problem.
It turns out that this problem is also geodesically convex after the introduction of a suitable metric on the cone of positive definite matrices; see \cite{AGLOW18}.
Further, \cite{AGLOW18}  also showed how to extend a method from convex optimization to this setting. 

A related problem is that of computing the Brascamp-Lieb constant~\cite{brascamp1976best} an important tool from functional analysis with application on convex geometry~\cite{ball1989volumes}, information theory~\cite{carlen2009subadditivity},\cite{LiuCCV16},\cite{LiuCCV17}, machine learning~\cite{hardt2013algorithms}, and theoretical computer science \cite{DSW14, DvirGOS16}.
At the moment, there is no known convex formulation of the Brascamp-Lieb constant.
However, recently, a geodesically convex formulation of the problem over the positive definite cone~\cite{vishnoi2018geodesically} was discovered; it remains open whether this formulation can be solved efficiently.

\paragraph{Organization of this exposition.}
The goal of this exposition is to introduce the reader with geodesic convexity and its applications. 
We do not assume any knowledge of differential geometry and present  the mathematical notions necessary to understand and formulate geodesics and geodesic convexity. 
We start by introducing differentiable and Riemannian manifolds in \cref{sec:manifolds}.
In \cref{sec:differentiation-on-manifolds}, we introduce the notion of differentiation of functions and vector fields on manifolds.
In particular we introduce the notion of an ``affine connection'' that allows us to take derivatives of one vector field on a manifold with respect to another.
We show that affine connections can be completely specified by a tensor whose entries are called ``Christoffel symbols''.
We also prove the existence of the uniqueness of the differential structure compatible with the Riemannian metric -- the ``Levi-Civita connection''.
This gives rise to two views of geodesics, one as straight lines on a manifold and second as length-minimizing curves on a manifold, which are discussed in \cref{sec:geodesic}.
The second view can be reformulated as the ``Euler-Lagrange dynamics'' on a manifold giving rise to the differential equations that enable us to compute geodesics.
Subsequently, we define geodesic convexity  in \cref{sec:g-convexity} and discuss various structural aspects of it.
Finally, we present geodesically convex formulation for the Brascamp-Lieb problem and the Operator Scaling Problem in \cref{sec:BL} and \cref{sec:OS}.
We do not present methods for geodesic convex optimization here.

\section{Manifolds}
\label{sec:manifolds}

Roughly speaking, a manifold is a  topological space that  resembles a Euclidean space  at each point.
This resemblance is a local property and it may be applicable only around a very small neighborhood of each point.

\begin{definition}[\bf Manifold]
A set $M\subseteq \R^n$ with an associated topology $T$ is a $d$-dimensional \emph{manifold} if for any point $p\in M$, there exists an open neighborhood of $p$ with respect to $T$, $U_p$, such that there exists a homeomorphism, $\phi_p$, between $U_p$ and an open set in $\R^d$ with respect to standard topology. 
The tuple $(U_p, \phi_p)$ is called a \emph{chart} and the collection $(U_p, \phi_p)$, for all $p \in M$, is called an \emph{atlas}.
\label{def:manifold}
\end{definition}
\noindent
\cref{fig:manifold_examples} presents some examples of local resemblance to the Euclidean spaces.

\begin{figure}[!htb]
\begin{center}
\includegraphics[scale=0.7]{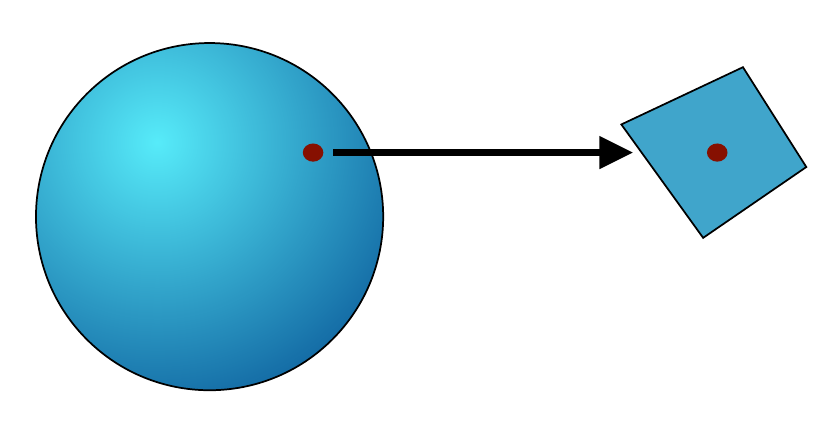}
\hspace{1cm}
\includegraphics[scale=0.7]{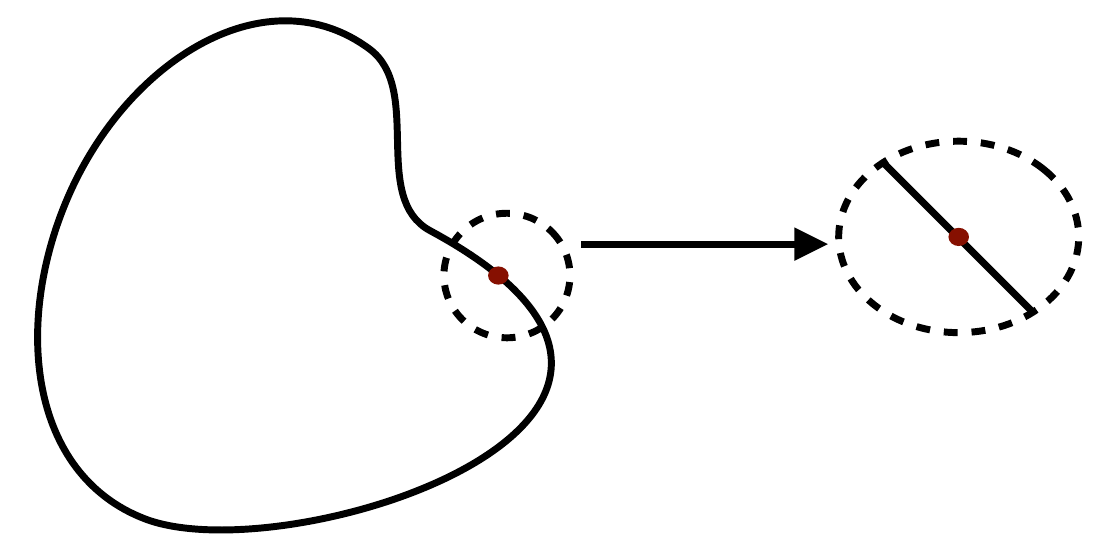}
\\ 
\includegraphics[scale=0.7]{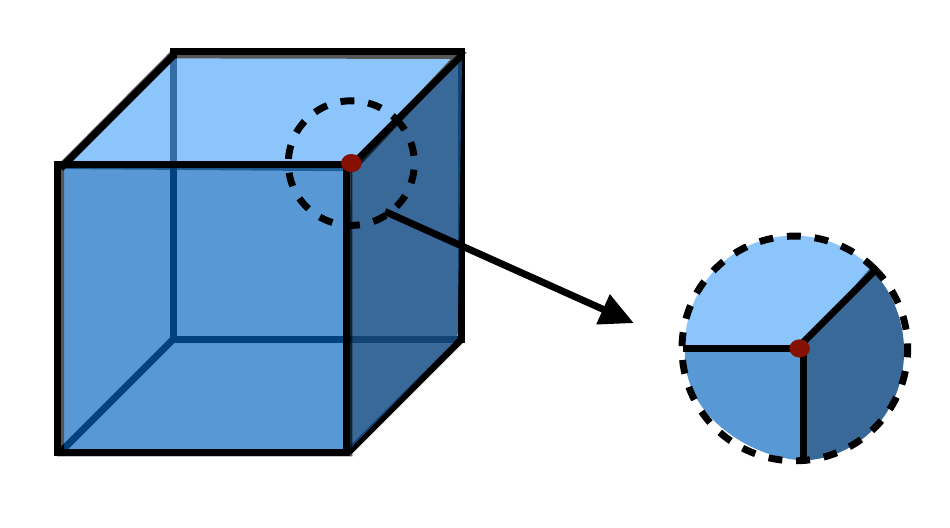}
\hspace{1cm}
\includegraphics[scale=0.7]{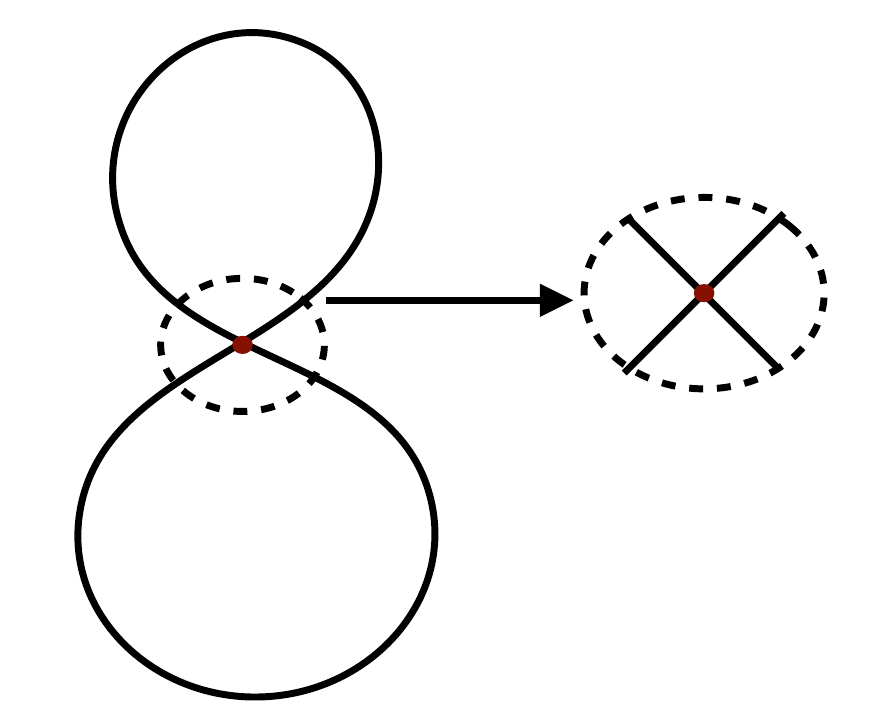}
\end{center}
\caption{The sphere (top left), the boundary of a cube (bottom left),  and the shape at the top right can all be considered manifolds since, in a small neighborhood around any point, they are all (topologically) Euclidean. The figure-eight (bottom right) is not a manifold since it is not Euclidean in any neighborhood of the point where the figure-eight self-intersects. 
\label{fig:manifold_examples}}
\end{figure}

\begin{example}[{\bf The Euclidean space}, $\R^n$]
$\R^n$ with the standard topology is a Euclidean space.
For any point $p\in \R^n$ and any open neighborhood of $p$, $U_p$, the homeomorphism $\phi_p$ is the identity map.
In the remainder, whenever we discuss $\R^n$, we use these charts.
\end{example}

\begin{example}[{\bf The positive definite cone}, $\Sn^n_{++}$]
The set of $n \times n$ real symmetric matrices, $\Sn^n$, with the standard topology is homeomorphic to $\R^{n(n+1)/2}$.
Let $\{E_{ij}\}_{i,j=1}^n$ be the standard basis elements for $n\times n$ matrices and $\{e_i\}_{i=1}^{n(n+1)/2}$ be the standard basis elements for $\R^{n(n+1)/2}$.
Let $\sigma$ be a bijective mapping from $\Set{(i, j)}{1\leq i\leq j\leq n}$ to $\{1,\hdots,n(n+1)/2\}$.
For convenience, let us also set $\sigma((i,j)):=\sigma((j,i))$ if $i>j$.
Then, the homeomorphism between $\Sn^n$ and $\R^{n(n+1)/2}$ is given by 
$$\phi\left(\sum_{i,j=1}^n \alpha_{ij}E_{ij}\right) = \sum_{i,j=1}^n \alpha_{ij}e_{\sigma((i,j))}.$$
For any point $p\in \Sn^n_{++}$ and open neighborhood of $p$, $U_p$, the homeomorphism $\phi_p$ is the restriction of $\phi$ to $U_p$.
In the remainder, whenever we discuss $\Sn^n_{++}$, we use these charts.
\end{example}

\subsection{Differentiable manifolds}

\noindent
In the most general sense, a  manifold is an abstract geometric object.
Although it has interesting topological properties by itself, in various applications it becomes important  to be able to do calculus over manifolds, just as in the Euclidean space.
The first step towards doing that is the requirement that the different charts be consistent.
If we have two charts corresponding to two neighborhood of a point $p$, then these charts should agree on the region they intersect.
We can think of this as a having two different maps for the same region and 
if we want to go from one place to another place that is covered by both maps, then it should not matter which map we are using.
Furthermore, if we are following a smooth path in one of these maps, then the corresponding path in the other map should also be smooth.
We formalize this intuition by defining differentiable manifolds.

\begin{definition}[\bf Differentiable manifold]
Let ${M}$ be a manifold,  and let $U_{\alpha}$ and $U_{\beta}$ be two open neighborhoods of $p\in {M}$ with charts $(U_\alpha,\,\psi_\alpha)$, $(U_\beta,\,\psi_\beta)$.
There is a natural bijection $\psi_{\alpha\beta} = \psi_\beta \circ \psi_\alpha^{-1}$ between $\psi_{\alpha}(U_{\alpha}\cap U_{\beta})$ and $\psi_{\beta}(U_{\alpha}\cap U_{\beta})$ called the transition map.
The manifold ${M}$ is called a \emph{differentiable manifold} if all transition maps are differentiable.
More generally, if all transition functions are $k$-times differentiable, then ${M}$ is a $C^k$-manifold.
If all transition functions have derivatives of all orders, then ${M}$ is  said to be a $C^\infty$ manifold or \emph{smooth manifold}.
\label{def:differentiable-manifold}
\end{definition}

\begin{figure}[!htb]
\centering
\includegraphics[keepaspectratio, width=0.7\textwidth]{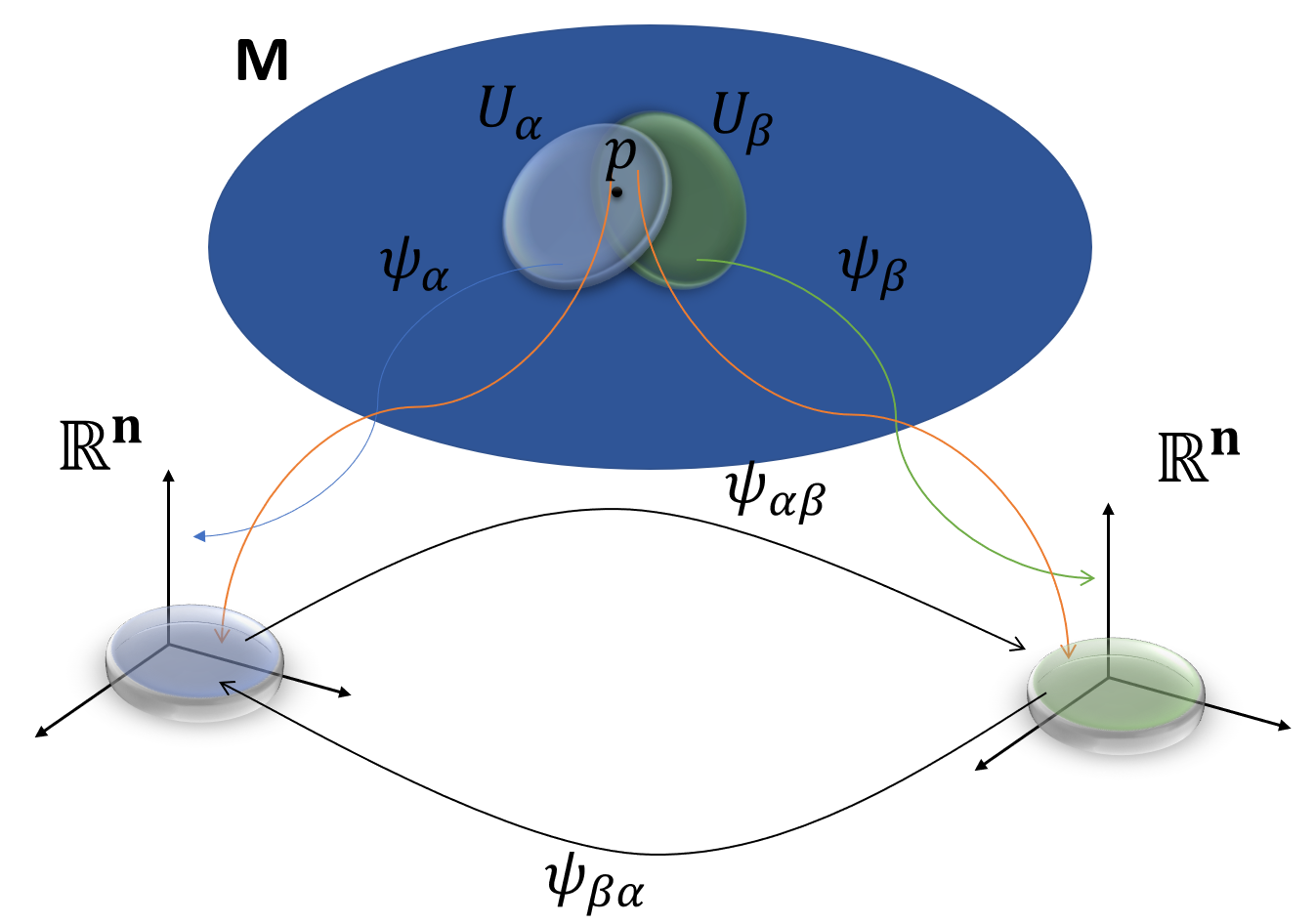}
\caption{A representation of transition maps.}
\label{fig:dif-manifold}
\end{figure}

\noindent
One of the main reasons we have defined differentiable manifolds is that we would like to move over the manifold.
However, when we are moving over the manifold we cannot leave it.
Depending on the manifold there can be some possible directions that leave the manifold.
For example, if we consider $d$-dimensional manifold $M\subseteq \R^n$ with $d<n$, then for any $p\in M$, there should be at least one direction that leaves the manifold.
If no such direction exists, then $M$ locally behaves like $\R^n$ and so it should be $n$-dimensional. 
We want to characterize possible directions that allows us to stay in the manifold after instantaneous movement in that direction.

\begin{definition}[\bf Tangent space]
Let $M$ be a $d$-dimensional differentiable manifold and $p$ be a point in $M$.
By definition, there exists a chart around $p$, $(U_p, \phi_p)$.
The \emph{tangent space} at $p$, $T_p M$ is the collection of directional derivatives of $f \circ \phi_p^{-1}:\phi_p(U_p)\to \R$ for any differentiable function $f:\R^d\to \R$.
The collection of $T_p M$ for $p$ in $M$ is called the \emph{tangent bundle} and denoted by $TM$.
\label{def:tangent-space}
\end{definition}

\noindent
One can ask how can we compute the directional derivatives of $f \circ \phi_p^{-1}$.
Observe that $V_p:=\phi_p(U_p)$ is a subset of $\R^d$ by the definition of charts.
Subsequently, note that $f \circ \phi_p^{-1}$ is simply a real-valued differentiable multivariable function.
Thus, it can be computed using standard multivariable calculus techniques.

Another important point is that we have defined a tangent space with respect to charts.
Hence, we need to ensure that different selections of charts lead to the same definition.
This can be seen by noting the consistency provided by the  manifold being a differentiable manifold.
When we do calculus over differentiable manifolds, often, we need to sum elements from a tangent space.
This makes sense only if a tangent space is actually a vector space.
\begin{proposition}[\bf The tangent space is a vector space]
Let $M$ be a differentiable manifold, and $p\in M$.
Then, $T_p M$ is a vector space. 
\end{proposition}
\begin{proof}
We need to verify two properties in order to show that $T_p M$ is a vector space.
If $u, v\in T_p M$ and $c\in \R$, then $cu \in T_p M$ and $u+v\in T_p M$.
\begin{itemize}
	\item If a vector $u\in T_p M$ is a directional derivative of $f\circ \phi_p^{-1}$ for some function $f$, then $cu$ is the directional derivative of $(cf)\circ \phi_p^{-1}$ for each $c\in \R$. Thus, $cu \in T_p M$.
	\item If another vector $v\in\R$ is a directional derivative of $f' \circ\phi_p^{-1}$ for some function $f'$, then $u+v$ is the directional derivative of $(f+f')\circ \phi_p^{-1}$. Thus $u+v\in T_p M$.
\end{itemize}
\end{proof}

\begin{figure}[!htb]
\centering
\includegraphics[keepaspectratio=true, width=0.5\textwidth]{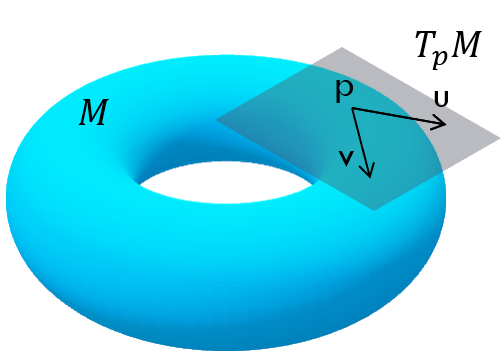}
\caption{The tangent space $T_pM$ to a  manifold $M$ at the point $p \in M$. \label{fig:tangent_space}}
\end{figure}

\begin{example}[\bf The Euclidean space, $\R^n$]
Let $p\in \R^n$ and $U_\alpha,U_\beta$ be two open neighborhoods of $p$ with charts $(U_\alpha, \phi_\alpha)$ and $(U_\beta, \phi_\beta)$.
As $\phi_\alpha$ and $\phi_\beta$ are identity maps, $\phi_{\alpha\beta}:=\phi_\beta\circ\phi_\alpha^{-1}$ is also an identity map on $U_\alpha\intersection U_\beta$.
Since an identity map is smooth, $\R^n$ is a smooth manifold.
The tangent space at $p\in\R^n$, $T_p M$, is $\R^n$ for each $p\in \R^n$.
We can verify this by considering functions $f_i(u):=\inner{e_i, u}$ where $e_i$ is the $i$th standard basis element for $\R^n$.
\end{example}

\begin{example}[{\bf The positive definite cone}, $\Sn^n_{++}$]
Let $P\in \Sn^n_{++}$ and $U_\alpha,U_\beta$ be two open neighborhoods of $p$ with charts $(U_\alpha, \phi_\alpha)$ and $(U_\beta, \phi_\beta)$.
As $\phi_\alpha$ and $\phi_\beta$ are restrictions of $\phi$ to $U_\alpha$ and $U_\beta$, $\phi_{\alpha\beta}:=\phi_\beta\circ\phi_\alpha^{-1}$ is also a restriction of $\phi$ to $U_\alpha\intersection U_\beta$.
Since $\phi$ is smooth its restrictions to open sets are smooth.
Thus, $\Sn^n_{++}$ is a smooth manifold.
The tangent space at $P\in\Sn^n_{++}$, $T_P M$, is $\R^{n(n+1)/2}$ which is homeomorphic to $\Sn^n$ for each $P\in \Sn^n_{++}$.
We can verify this by considering functions $f_i(u):=\inner{e_i, u}$ where $e_{i}$ is the $i$th standard basis element for $\R^{n(n+1)/2}$.
To simplify our calculations we take tangent spaces as $\Sn^n$, the set of all $n \times n$ symmetric matrices.
\end{example}

\noindent
So far,  we have defined a differential manifold and tangent spaces that consist of collection of directions we can follow to move around the manifold.
Now, we introduce vector fields to describe our movement over the manifold.

\begin{definition}[\bf Vector field]
A vector field $X$ over a $d$-dimensional differentiable manifold $M$ is an assignment of tangent vectors to points in $M$.
In other words, if $p\in M$, then $X(p)$ is a vector from $T_p M$.
If $f$ is a smooth real-valued function on $M$, then multiplication of $f$ with $X$ is another vector field $fX$ defined as follows, 
\begin{equ}
(f X) (p) := f(p) X(p).
\end{equ}
We say $X$ is a smooth vector field if, for any coordinate chart $\phi_p$ with basis elements $e_1,\hdots, e_d$ for $T_p M$, there are smooth functions $f_1,\hdots, f_d$ such that 
$$X(q)=\sum_{i=1}^d f_i(\phi_p(q))e_i$$ for points in the neighborhood of $p$.
\end{definition}

\begin{figure}[!htb]
\centering
\includegraphics[keepaspectratio, width=0.4\textwidth]{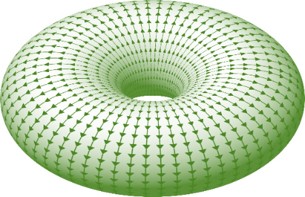} \hspace{1cm}
\includegraphics[keepaspectratio, width=0.4\textwidth]{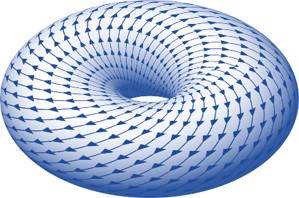}
\caption{Two different vector fields over a torus.}
\label{fig:vector-fields}
\end{figure}

\noindent
Let us consider the problem of moving from a point $p$ to another point $q$.
If $p$ and $q$ are not close to each other than there will be no chart covering both points.
Consequently, it is not clear in which direction we move from $p$ to reach $q$.
We use vector fields to provide a complete description of a path, or curve, from $p$ to $q$.
In a nutshell, a vector field tells us which direction we need to move given our current position.

Lastly, we define frame bundle: a tuple of ordered vector fields that forms an ordered basis for tangent spaces.
\begin{definition}[\bf Frame bundle]
Let $M$ be a $d$-dimensional smooth manifold.
A frame bundle on $M$, $\{\partial_i\}_{i=1}^d$ is an ordered tuple of $d$ smooth vector fields, such that for any $p\in M$, $\{\partial_i(p)\}_{i=1}^d$ form a basis for $T_p M$.
\label{def:frame-bundle}
\end{definition}

\noindent
We remark that, although we use the symbol $\partial_i$ which is the usual symbol for partial derivatives, frame bundles are different objects than partial derivatives.

\subsection{Riemannian manifolds}
Differential manifolds allow us to do calculus that is necessary for optimization.
However, as such, they lack quantitative notions such as distances and angles.
For example, when we do optimization and are interested in optimal solutions of the problem rather than optimal value, we need to approximate points.
Without a measure of distance, we cannot argue about approximating points.
Towards this end, we introduce Riemannian manifolds, differential manifolds on which one can measure  distances, lengths, angles etc.
Before we introduce how to measure distance on a manifold, let us recall how we measure distance on the Euclidean space.
Euclidean space is a metric space with metric $d:\R^n\times\R^n\to\R^n$, $d(p,q)=\norm{p-q}_2$.
This metric gives the distance between $p$ and $q$ which is also the length of the straight curve joining $p$ and $q$.
If we examine this metric carefully, we  notice that $d$ actually computes the norm of the $q-p$, the direction we follow to reach $q$.
This direction vector is a part of the tangent space at each point over the line segment joining $p$ to $q$.
As we saw earlier, this tangent space is simply $\R^n$, itself.
Thus, what we need is a way to define norms of vectors in a tangent space or, more generally, inner products between two vectors in a tangent space.
This is what the metric tensor does.

\begin{definition}[\bf Metric tensor]
Let $M\subseteq \R^n$ be a smooth manifold.
A \emph{metric tensor} over $M$ is a collection of functions $g_p:T_p M\times T_p M \to \R$ for $p\in M$ that satisfies following conditions,
\begin{enumerate}
	\item[i.] $g_p$ is a symmetric function: $g_p(u, v) = g_p(v, u)$ for all $u, v\in T_p M$,
	\item[ii.] $g_p$ is linear in the first argument: $g_p(cu+v,w) = cg_p(u,w)+g_p(v,w)$ for all $u, v, w\in T_p M$ and $c\in \R$, and
\end{enumerate}
\vspace{2mm}
$g$ is called \emph{non-degenerate}, if for a fixed $0\neq u\in T_p M$, $g_p(u, v)$ is not identically zero, in other words $\exists v\in T_p M$ such that $g_p(u, v)\neq 0$.

\vspace{2mm}
\noindent
$g$ is called \emph{positive definite}, if 
$$g_p(v,v)\geq 0$$ for $v\in T_p M$ and $g_p(v, v)= 0$ if and only if $v=0$.
\label{def:metric-tensor}
\end{definition}

\noindent
In the Euclidean space, the metric is the same for every tangent space, however, this is not required for  general manifolds.
This definition allows us to assign arbitrary metrics to individual tangent spaces.
But this is neither intuitive nor useful.
We know that tangent spaces change smoothly in a small neighborhood around a point.
As a result, tangent vectors also change smoothly in the same neighborhood.
Our expectation is that the distance between these tangent vectors also changes smoothly.
A class of metric tensors that satisfy this condition give rise to Riemannian manifolds.

\begin{definition}[\bf Riemannian manifold]
Let $M$ be a smooth manifold and $g$ be a metric tensor over $M$.
If for any open set $U$ and smooth vector fields $X$ and $Y$ over $U$, the function $g(X, Y)[p] := g_p(X_p, Y_p)$ is a smooth function of $p$, and $g$ is positive definite, then $(M, g)$ is called a \emph{Riemannian manifold}.
\label{def:riemannian-manifold}
\end{definition}

 \noindent
If we relax the positive definiteness condition with just non-degenerateness, then one gets what is called a pseudo-Riemannian manifold.
Pseudo-Riemannian manifolds arise  in the theory of relativity and, although we focus only on Riemannian manifolds, most of the results can be extended to pseudo-Riemannian manifolds.
One critical difference between pseudo-Riemannian manifolds and Riemannian manifolds is that the local norm of any tangent vector is non-negative in a Riemannian manifold, while local norm of some tangent vectors can be negative on a pseudo-Riemannian manifold. 

\begin{example}[{\bf The Euclidean space}, $\R^n$]
The usual metric tensor over $\R^n$ is $g_p(u, v):=\inner{u, v}$ for $x, u, v\in\R^n$.
$g_p$ is clearly a symmetric, bilinear, and positive definite function.
It is also non-degenerate as $\inner{u, v}=0$ for every $v$ implies $\inner{u, u}=0$ or equivalently $u=0$.
Next, we observe that $\R^n$ with $g$ is a Riemannian manifold.
This follows from the observation
\begin{equ}
g(X,Y)[p]:=g_p(X_p, Y_p) = \inner{X_p, Y_p}
\end{equ}
is a smooth function of $p$ as $X$ and $Y$ are smooth vector fields.
In the remainder, whenever we talk about $\R^n$ we use this metric tensor.
\end{example}

\begin{example}[{\bf The positive orthant}, $\R^n_{>0}$]
We consider the metric tensor induced by the Hessian of the log-barrier function: $-\sum_{i=1}^n \log(x_i)$.
The induced metric tensor is 
$$g_p(u, v):= \inner{P^{-1}u, P^{-1} v}$$
 where $P$ is a diagonal matrix whose entries is $p$ for $p\in\R^n_{>0}$ and $u,v\in\R^n$.
$g_p$ is clearly a symmetric,  bilinear, and positive definite function.
It is also non-degenerate as $\inner{P^{-1}u, P^{-1}v}=0$ for every $v$ implies $\inner{P^{-1}u, P^{-1}u}=0$ or equivalently $P^{-1}u=0$.
Since $P$ is a non-singular matrix, $P^{-1}u=0$ is equivalent to $u=0$.
Next, we observe that $\R^n_{>0}$ with $g$ is a Riemannian manifold.
This follows from the observation
\begin{equ}
g(X,Y)[p]:=g_p(X_p, Y_p) = \inner{P^{-1}X_p, P^{-1}Y_p}=\sum_{i=1}^n \frac{X_p^i Y_p^i}{p_i^2}
\end{equ}
is a smooth function of $p$ where $X_p:=\sum_{i=1}^nX_p^i e_i$ and $Y_p:=\sum_{i=1}^n Y_p^i e_i$ with $(e_i)_{i=1}^n$ standard basis elements of $\R^n$.
Here we used the fact that the tangent space of $\R^n_{>0}$ is $\R^n$.
Now, $X_p^i$, $Y_p^i$, and $p_i^{-2}$ are all smooth functions of $p$.
Thus their sum, $g(X, Y)[p]$ is a smooth function of $P$.
In the remainder of this note, whenever we talk about $\R^n_{>0}$ we use this metric tensor.
\end{example}

\begin{remark}[\bf Hessian manifold] 
In fact, when a metric arises as a Hessian of a convex function, as in the example above, it is called a {\em Hessian manifold}.
\end{remark}

\begin{example}[{\bf The positive definite cone}, $\Sn^n_{++}$]
We consider the metric tensor induced by the Hessian of the log-barrier function: $-\log\det(P)$ for a positive definite matrix $P$.
The induced metric tensor is 
$$g_P(U, V):= \tr[P^{-1}UP^{-1} V]$$
 for $P\in\Sn^n_{++}$ and $U,V\in\Sn^n$.
$g_P$ is clearly a symmetric, bilinear, and positive definite.
It is also non-degenerate as $\tr[P^{-1}UP^{-1} V]=0$ for every $V$ implies 
\begin{equ}
\tr[P^{-1}UP^{-1} U]=\tr[P^{-1/2}UP^{-1/2} P^{-1/2}UP^{-1/2}]=0
\end{equ}
 or equivalently $P^{-1/2}UP^{-1/2}=0$.
Since $P$ is a non-singular matrix, $P^{-1/2}UP^{-1/2}=0$ is equivalent to $U=0$.
Next, we observe that $\Sn^n_{++}$ with $g$ is a Riemannian manifold.
This follows from the observation
\begin{equ}
g(X,Y)[P]:=g_P(X_P, Y_P) = \tr[P^{-1}X_P P^{-1}Y_P]
\end{equ}
is a smooth function of $P$.
Our observation is based on $P^{-1}$, $X_P$, and $Y_P$ are all smooth functions of $P$.
Thus, their multiplication $P^{-1}X_P P^{-1} Y_P$ is a smooth function.
Finally, trace as a linear function is smooth.
Thus, $g(X,Y)[P]$ is a smooth function of $P$.
In the remainder of this note, whenever we talk about $\Sn^n_{++}$ we use this metric tensor.
\end{example}

\section{Differentiation on Manifolds}
\label{sec:differentiation-on-manifolds}
Consider the task of optimizing a  given real valued function $f$ over a manifold. 
In a Euclidean space, traditional methods such as gradient descent or Newton's method move from a point to another in the direction of the negative gradient of $f$ in order to minimize $f$.
The performance of such methods also depends on how the derivative of $f$ with respect to a direction  behaves.
A smooth manifold, on the other hand, while allows us to define smooth ``curves'', does not tell us how to differentiate functions with respect to vector fields or, more generally, how to differentiate a vector field with respect to another.
Differentiation inherently requires comparing the change in a function or vector value when we move from one point to another close by point (in a given direction), but the fact that at different points the charts are different poses a challenge in defining differentiation.
Additionally, it turns out that there is no unique way to do differentiation on a smooth manifold and a differential structure that satisfies some natural properties that we would expect is referred to as an ``affine connection''. 
In this section, we present the machinery from differential geometry that is required to understand how to  differentiate  over a differentiable manifold.
On Riemannian manifolds, if we impose additional constraints on the affine connection, compatibility with the metric tensor, then this gives rise to the Levi-Civita connection that is shown to be unique.
This section is important to define what it means for a curve on a manifold  to be a geodesic or a straight line which, in turn, is important to define geodesic convexity.
 
\subsection{Differentiating a function on a manifold}
Let us recall how does one differentiate real-valued functions on the Euclidean space.
Let $f:\R^n\to \R$ be a differentiable function.
The differential of $f$ at $x\in\R^n$, $df(x)$ is the unique vector that satisfies,
\begin{equ}[eq:usual-differentiation]
\lim_{h\to 0} \frac{f(x+h) - f(x) - \inner{df(x), h}}{\norm{h}} = 0
\end{equ}
The direction vector $h$ in~\eqref{eq:usual-differentiation} is an arbitrary vector in $\R^n$.
In the manifold setting, we cannot approach to given point from arbitrary directions.
The directions are limited to the tangent space at that point.
On the other hand, there exists a homeomorphism between the tangent space and some Euclidean space.
Thus, we can apply usual differentiation operator to functions over manifolds with slight modifications.

\begin{definition}[\bf Derivative of a function]
Let $M$ be a differentiable manifold, $p\in M$, $T_p M$ be the tangent space at $p$ and $(U_p, \phi_p)$ be a chart around $p$.
Let $f:M\to \R$ be a continuous function.
We say $f$ is differentiable at $p$, if there exists a linear operator $\lambda_p:T_p M \to \R$ such that
\begin{equ}[eq:manifold-differentiation]
	\lim_{h\to 0} \frac{f(\phi_p^{-1}(\phi_p(p)+h)) - f(p) - \lambda_p(h)}{\norm{h}} = 0,
\end{equ}
where $h$ varies over $T_p M$.
\end{definition}

\noindent
In the example manifolds we have discussed earlier, manifolds and their tangent spaces lie in a common Euclidean space.
Thus, we can drop the mappings $\phi_p$ from~\eqref{eq:manifold-differentiation}.
Consequently, differentiation on $\R^n$ and $\R^n_{>0}$ is simply the usual differentiation.\footnote{Note,  however, that this is not true if we were considering derivation along a geodesic or an arbitrary curve in the positive orthant example.}
On the other hand, we restrict vectors $h$ to $\Sn^n$ in~\eqref{eq:usual-differentiation} while differentiating a function on $\Sn^n_{++}$.

\subsection{Differentiation a function along a vector field on a manifold}
Differentiation along a vector field is a generalization of partial derivatives from multivariable calculus.
The vector field gives a direction that is tangent to the manifold at every point on the manifold.
Differentiating $f$ along the vector field $X$ is equivalent to computing directional derivatives of $f$ at every point in the direction given by $X$.

\begin{definition}[\bf Derivative of a function along a vector field]
Let $M$ be a differentiable manifold, $X$ be a vector field over $M$, $p\in M$, $T_p M$ be the tangent space at $p$ and $(U_p, \phi_p)$ be a chart around $p$.
Let $f:M\to \R$ be a differentiable function.
We define the derivative of $f$ along $X$ at $p$ by 
\begin{equ}[eq:derivative-along-vector-field]
X(f)(p) := \lim_{h\to 0}\frac{f(\phi_p^{-1}(\phi_p(p)+h X(p))) - f(p)}{h}.
\end{equ}
\end{definition}

\noindent
Similar to differentiation of a function on the manifold, we can drop the mappings $\phi_p$ from~\eqref{eq:derivative-along-vector-field} for the example manifolds we have discussed so far.
Thus, differentiation along a vector field on these manifolds is exactly same as usual directional derivativation.
The only difference is that, at any point, the vector field specifies the direction of the directional derivative.

\subsection{Affine connections and Christoffel symbols}
We saw how to differentiate a function over a manifold as well as how to differentiate a function along a given vector field.
However, both of these definitions depend on the local charts and do not tell us how to compare different curves between ``far away'' points.
In order to compare different curves, we need to be able to compare vector fields inducing these curves.
For that reason, we would like to measure the relative movement of a vector field $Y$ with respect to another vector field $X$.
The idea is very similar to the second order partial derivatives, where we measure the change in one component with respect to another component.
Consequently, if we denote the collection of vector fields over $M$ by $\mathfrak{X}(M)$, then we would like to define an operator  
$$\nabla:\mathfrak{X}(M)\times\mathfrak{X}(M)\rightarrow\mathfrak{X}(M)$$ called the affine connection so that $\nabla$ measures the relative movement between vector fields and behaves like a differential operator.
What we mean by behave like a differential operator is that it should be bilinear and it should satisfy Leibniz's rule.

\begin{definition}[\bf Affine connection]
Let $M$ be differentiable manifold and $\mathfrak{X}(M)$ be the collection vector spaces over $M$.
An operator $\nabla:\mathfrak{X}(M)\times \mathfrak{X}(M)\to\mathfrak{X}(M)$ is called \emph{affine connection} if it satisfies following conditions,
\begin{itemize}
	\item {\bf linearity in the first term}: for any smooth functions $f,f':M\to\R$ and $X,Y,Z\in\mathfrak{X}(M)$,
	\begin{equ}[eq:affine-1]
		\nabla_{fX+f'Y}Z = f\nabla_X Z + f'\nabla_Y Z.
	\end{equ}
	\item {\bf linearity in the second term}: for any $X, Y, Z\in\mathfrak{X}(M)$,
	\begin{equ}[eq:affine-2]
		\nabla_X (Y+Z) = \nabla_X Y + \nabla_X Z.
	\end{equ}
	\item {\bf Leibniz's rule}: for any smooth function $f:M\to\R$ and $X,Y\in\mathfrak{X}(M)$,
	\begin{equ}[eq:affine-3]
		\nabla_X (fY) = f\nabla_X Y + Y X(f).
	\end{equ}
\end{itemize}
\label{def:affine-connection}
\end{definition}

\noindent
Such mappings are called ``affine connections" because they allow one to describe tangent vectors of two different points, in some sense ``connecting" two points on the manifold.
An affine connection at a point $p$ is actually a linear operator from $T_p M\times T_p M$ to $T_p M$ and it can be represented with a $d\times d\times d$ tensor.
Christoffel symbols of the second kind correspond to the entries of these tensors.

\begin{definition}[\bf Christoffel symbols of the second kind]
Let $M$ be a $d$-dimensional differentiable manifold and $TM$ be its tangent bundle.
Let us fix one frame $\{\partial_i\}_{i=1}^d$ for $M$.
Given an affine connection $\nabla$, define  
\begin{equ}[eq:christoffel-symbol]
\Gamma_{ij}^k := (\nabla_{\partial_i} \partial_j)^k,
\end{equ}
the coefficient of $\nabla_{\partial_i} \partial_j$ in the direction of $\partial_k$.
The coefficient $\Gamma_{ij}^k$ is called a \textit{Christoffel symbol of the second kind}.
\label{def:christoffel}
\end{definition}

\begin{theorem}[\bf Christoffel symbols determine the affine connection]
The Christoffel symbols of the second kind determine the affine connection $\nabla$ uniquely.
\label{thm:christoffel-symbols}
\end{theorem}

\begin{proof}
Let $X,Y$ be two vector fields on $M$ and $\{\partial_i\}_{i=1}^d$ be the frame bundle on $M$.
Let us write $X$ as $\sum_{i=1}^d X^i \partial_i$ and $Y$ as $\sum_{i=1}^d Y^i\partial_i$.
Then, $\nabla_X Y$ can be written as
\begin{align*}
\nabla_X Y 
	\stackrel{\eqref{eq:affine-1}}{=}& \sum_{i=1}^d X^i \nabla_{\partial_i} Y\\
	\stackrel{\eqref{eq:affine-2}}{=}& \sum_{i,j=1}^d X^i \nabla_{\partial_i} Y^j\partial_j\\
	\stackrel{\eqref{eq:affine-3}}{=}& \sum_{i,j=1}^d X^iY^j \nabla_{\partial_i} \partial_j + X^i \partial_j \partial_i(Y^j)\\
	=& \sum_{i,j,k=1}^d X^iY^j\Gamma_{ij}^k \partial_k + \sum_{i,j=1}^d X^i \partial_j \partial_i(Y^j)\\
	=& \sum_{k=1}^d \left(\sum_{i,j=1}^d \Gamma_{ij}^k + \sum_{i=1}^d X^i \partial_i(Y^k)\right)\partial_k.
\end{align*}
We note that $\sum_{i=1}^d X^i \partial_i(Y^k)$ neither depends on the affine connection nor the Christoffel symbols of the second kind for any $k=1,\hdots,d$.
It depends only to the vector fields $X$, $Y$, and the tangent spaces induced by the topology over the manifold.
Hence, given the Christoffel symbols of the second kind, there is unique affine connection corresponding to these symbols.
\end{proof}

\subsection{Derivative of a vector field along a smooth curve}
We already saw how to compute the derivative of a curve along a vector field as well how to differentiate a vector field with respect to another vector field.
Now, we discuss how to compute derivatives of a vector field along a curve.
A useful application of this is that one can move a tangent vector along the curve so that the tangent vector stays parallel with respect to the curve -- {\em parallel transport} -- that we do not discuss in this exposition.

Let us consider a smooth curve $\gamma:I\to M$ on differentiable manifold $M$ for an interval $I\subseteq \R$.
$\dot{\gamma}$ describes a tangent vector for each $\gamma(t)\in M$ for $t\in I$.
However, $\dot{\gamma}$ may not correspond to a vector field, in general.
A vector field needs to assign a unique tangent vector at any given point.
However, if a smooth curve, $\gamma:I\to M$, intersects by itself, then the curve defines two tangent vectors at the intersection.
Unless these tangent vectors are same, there cannot be a vector field $X$ such that 
$$\forall t\in I, \;\;  X(\gamma(t)) = \dot{\gamma}(t).$$
Although we cannot find  a vector field whose restriction to $\gamma$ is $\dot{\gamma}$ globally, we can still define differentiation with respect to $\gamma$ that is compatible with the affine connection $\nabla$.
The differentiation of a vector field $X$ with respect to a smooth curve $\gamma$ measures the change of $X$ along $\dot{\gamma}$  and we denote it by $\nabla_{\dot{\gamma}} X$.
\begin{definition}[\bf Covariant derivative]
Let $M$ be a differentiable manifold and $\nabla$ be an affine connection on $M$.
Let $\gamma:I\to M$ be a smooth curve for some interval $I\subseteq \R$.
We say $\nabla_{\dot{\gamma}}$ is a {\em covariant derivative} if for any vector fields $X$ and $Y$ on $M$, real constants $\alpha, \beta$ and smooth function $f:M\to\R$,
\begin{align}
\nabla_{\dot{\gamma}} (\alpha X + \beta Y) =& \alpha \nabla_{\dot{\gamma}} X + \beta \nabla_{\dot{\gamma}} Y\qquad(\text{$\R$-linearity})\label{eq:covariant-linear}\\
\nabla_{\dot{\gamma}} fX =& \dot{f}X + f\nabla_{\dot{\gamma}} X\qquad(\text{Leibniz rule})\label{eq:covariant-leibniz}
\end{align}
and if there exists a vector field $Z$ such that $Z(\gamma(t))=\dot{\gamma}(t)$, then
\begin{equ}[eq:covariant-extension]
(\nabla_Z X)(\gamma(t)) = (\nabla_{\dot{\gamma}} X)(t).
\end{equ}
\end{definition}

\noindent
If we fix the curve $\gamma$, then there exists a unique operator that satisfies conditions~\eqref{eq:covariant-linear} and~\eqref{eq:covariant-leibniz}.

\begin{proposition}[{\bf Uniqueness of the covariant derivative}]
Let $M$ be a differentiable manifold, $\nabla$ be an affine connection on $M$ and $\gamma:I\to M$ be a smooth curve for some interval $I\subseteq \R$.
There is a unique covariant derivative corresponding to $\gamma$, denoted by $\nabla_{\dot{\gamma}}$.
\label{prop:covariant-unique}
\end{proposition}
\begin{proof}
Let $\{\partial_i\}_{i=1}^d$ be a frame bundle for $M$.
We consider $\partial_i$ as a function of $t$, $\partial_i(t):=\partial_i(\gamma(t))$.
For a given vector field $X$, let us write $X:=\sum_{i=1}^d X^i\partial_i$ and $\dot{\gamma}:=\sum_{i=1}^d \dot{\gamma_i}\partial_i$.
Then,
\begin{align*}
\nabla_{\dot{\gamma}} X
	=& \nabla_{\dot{\gamma}} \sum_{i=1}^d X^i \partial_i\\
	\stackrel{\eqref{eq:covariant-linear}}{=}& \sum_{i=1}^d \nabla_{\dot{\gamma}} X^i \partial_i\\
	\stackrel{\eqref{eq:covariant-leibniz}}{=}& \sum_{i=1}^d \dot{X}^i \partial_i + X^i \nabla_{\dot{\gamma}}\partial_i\\
	\stackrel{\eqref{eq:covariant-extension}}{=}& \sum_{i=1}^d\dot{X}^i \partial_i +X^{i}\nabla_{\sum_{j=1}^d\dot{\gamma}_j \partial_j}\partial_i\\
	\stackrel{\eqref{eq:affine-1}}{=}& \sum_{i,j=1}^d\dot{X}^i \partial_i + \dot{\gamma}_j X^{i}\nabla_{\partial_j}\partial_i\\
	=&\sum_{k=1}^d\left(\dot{X}^k+\sum_{i,j=1}^d\dot{\gamma}_iX^j\Gamma_{ij}^k\right)\partial_k
\end{align*} 
Thus, given an affine connection, there can be at most one covariant derivative corresponding to $\gamma$.
On the other hand, given $\nabla$, $\gamma$ and $X$, if we define $\nabla_{\dot{\gamma}} X$ as
\begin{equ}
\sum_{k=1}^d\left(\dot{X}^k+\sum_{i,j=1}^d\dot{\gamma}_iX^j\Gamma_{ij}^k\right)\partial_k
\end{equ}
then it satisfies the conditions for the covariant derivative.
Linearity clearly holds, and if $f:M\to \R$ is a smooth function, then
\begin{align*}
\nabla_{\dot{\gamma}} fX 
	=&	\sum_{k=1}^d\left(\left(\frac{d}{dt}fX^k\right)+\sum_{i,j=1}^d\dot{\gamma}_i(fX)^j\Gamma_{ij}^k\right)\partial_k\\
	=&	\sum_{k=1}^d\left(\dot{f}X^k+f\dot{X}^k+f\sum_{i,j=1}^d\dot{\gamma}_iX^j\Gamma_{ij}^k\right)\partial_k\\
	=&	\dot{f}X + f\nabla_{\dot{\gamma}} X. 
\end{align*}
Hence, covariant derivative exists.
\end{proof}

\subsection{The Levi-Civita connection}
Affine connections are very generic and there can be many of them.
In fact, for any $d\times d\times d$ tensor, there is a corresponding affine connection as stated in~\cref{thm:christoffel-symbols}.
In general, we do not want to consider every affine connection but those with ``nice'' properties.
For a Riemannian manifold, ``compatibility'' with the associated metric is one ``nice'' property we can ask.
By compatibility we mean the inner product between two vectors in tangent spaces should not change when we move around the manifold.
If we require the connection to satisfy another property called ``torsion-freeness'', then there is a unique connection that satisfies both properties.
This connection is called the Levi-Civita connection. 

We define ``torsion-freeness'' using Lie Bracket of vector fields.

\begin{definition}[\bf Lie Bracket]
Let $X$ and $Y$ be vector fields defined over a $d$-dimensional differentiable manifold $M$.
Let us fix a frame $\{\partial_i\}_{i=1}^d$ for $M$ and let us write $X=\sum_{i=1}^d X^i \partial_i$, $Y=\sum_{i=1}^d Y^i \partial_i$.
The Lie Bracket of $X$ and $Y$, $[X,Y]$ is
\begin{equ}
\;[X,Y] = \sum_{i,j=1}^d  (X^j \partial_j Y^i - Y^j \partial_j X^i) \partial_i.
\end{equ}
\label{def:lie-bracket}
\end{definition}

\noindent
Thus, we can think of a Lie Bracket as a differential operator induced by two smooth vector fields.
Given a real-valued smooth function $f$, and vector fields $X, Y$, we apply Lie bracket to $f$ as follows,
\begin{equ}
\;[X,Y](f):=\sum_{i,j=1}^d  X^j \partial_j (Y^i\partial_i(f)) - Y^j \partial_j(X^i \partial_i(f)).
\end{equ}
Linearity of this operator is trivial.
And with a quick calculation one can show that $[X,Y]$ satisfies the Leibniz rule, in other words for real-valued smooth functions $f$ and $g$,
\begin{equ}
\;[X,Y](fg)=f[X,Y](g)+[X,Y](f)g.
\end{equ}

\noindent
The torsion-free property requires that for every vector field pair $X$ and $Y$, the difference between the vector fields induced by derivation of $Y$ along $X$ and derivation of $X$ along $Y$ is equal to the differential operator induced by these vector fields.
Now, we can define the Levi-Civita connection formally.

\begin{definition}[\bf Levi-Civita connection]
Let $(M, g)$ be a Riemannian manifold.
An affine connection on $M$ is called Levi-Civita connection if it is torsion-free
\begin{equ}[eq:property1]
\forall X, Y\in\mathfrak{X}(M),\;\nabla_X Y - \nabla_Y X = [X,Y],
\end{equ}
and it is compatible with the metric $g$,
\begin{equ}[eq:property2]
\forall X,Y, Z \in \mathfrak{X}({M}), \;X(g(Y,Z)) = g(\nabla_X Y, Z) + g(Y, \nabla_X Z).
\end{equ}
\end{definition}

\noindent
The definition of Levi-Civita connection states that it is any affine connection that is both torsion-free and compatible with the metric.
From this definition, it is hard to see if Levi-Civita connection unique or even exists.
The fundamental theorem of Riemannian geometry asserts that Levi-Civita connection exists and unique.
This theorem also shows that the Christoffel symbols of the second kind, and hence the Levi-Civita connection, are completely determined by the metric tensor.
This property is useful in deriving expressions for geodesics in several cases.
In the scope of this exposition, we focus on the uniqueness only.
We present one Levi-Civita connection for the examples we have discussed which is the only Levi-Civita connection by the theorem.
The proof of this theorem appears in~\cref{sec:levi-civita-proof}.

\begin{theorem}[{\bf Fundamental Theorem of Riemannian Geometry}]
Let $(M,g)$ be a $d$-dimensional Riemannian manifold.  
The Levi-Civita connection on $(M, g)$ is exist and unique.
Furthermore, its Christoffel symbols can be computed in terms of $g$.
Let 
$$g_{ij}:=g(\partial_i, \partial_j),$$ $G$ be the $d\times d$ matrix whose entries are $g_{ij}$, and $g^{ij}$ be the $ij$th entry of $G^{-1}$.
The Christoffel symbols of Levi-Civita connection is given by following formula,
\begin{equ}
\Gamma_{ij}^k:=\frac{1}{2} \sum_{l=1}^d g^{kl}(\partial_i g_{jl} + \partial_{j} g_{il} - \partial_{l} g_{ij})
\end{equ}
for $i,j,k=1,\hdots,d$.
\label{thm:fundamental}
\end{theorem}

\begin{example}[{\bf The Euclidean space}, $\R^n$] 
Let us fix $\partial_i(p)$ as $e_i$, $i$th standard basis element of $\R^n$.
We can do this as $T_p M$ is $\R^n$ for any $p\in  \R^n$.
Also, $g_{ij}(p):=g_p(e_i, e_j)=\delta_{ij}$ where $\delta$ is Kronecker delta.
Since $g_{ij}$ is a constant function $\partial_k g_{ij}$, the directional derivative of $g_{ij}$ along $e_k$ is 0.
Consequently, $\Gamma_{ij}^k=0$ for any $i,j,k=1,\hdots,n$.
\label{ex:Rn-christoffel}
\end{example}

\begin{example}[{\bf The positive orthant}, $\R^n_{>0}$]
Let us fix $\partial_i(p)$ as $e_i$, $i$th standard basis element of $\R^n$.
We can do this as $T_p M$ is $\R^n$ for any $p\in  \R^n$.
Also, $g_{ij}(p):=g_p(e_i, e_j)=\delta_{ij}p_{i}^{-2}$ where $\delta$ is Kronecker delta.
Consequently, $G(p) = P^{-2}$ and $G(p)^{-1} =P^2$.
Equivalently $g^{ij}(p)=\delta_{ij} p_i^2$.
The directional derivative of $g_{ij}$ along $e_k$ at $p$, $\partial_k g_{ij}(p)$ is $-2\delta_{ij}\delta_{ik}p_i^{-3}$.
Consequently, $\Gamma_{ii}^i(p)=p_i^{-1}$ for $i=1,\hdots, n$ and $\Gamma_{ij}^k=0$ if all three $i,j$, and $k$ are not same.
\label{ex:po-christoffel}
\end{example}

\section{Geodesics}
\label{sec:geodesic}

\subsection{Geodesics as straight lines on a manifold}
Geometrically, geodesics are  curves whose tangent vectors remain parallel to the curve with respect to the affine connection.

\begin{definition}[\bf Geodesics on a differentiable manifold]
Let $M$ be a differentiable manifold and $\nabla$ be an affine connection on $M$.
Let $\gamma:I\to M$ be a smooth curve on $M$ where $I\subseteq \R$ is an interval.
$\gamma$ is a geodesic on $M$ if 
\begin{equ}
\nabla_{\dot{\gamma}} \dot{\gamma} = 0
\end{equ}
where $\dot{\gamma}:=\frac{d \gamma(t)}{dt}$.
\end{definition}

\noindent
We already saw that one can specify an affine connection in terms of its Christoffel symbols.
Similarly, we can also describe geodesics using Christoffel symbols.
This characterization is useful when we want to compute geodesics when the affine connection is given in terms of its Christoffel symbols.
An important example is the Levi-Civita connection whose Christoffel symbols can be derived from the metric tensor.
\begin{proposition}[\bf Geodesic equations in terms of Christoffel symbols]
Let $M$ be a $d$-dimensional differentiable manifold, $\{\partial_i\}_{i=1}^d$ be a frame of $M$, $\nabla$ be an affine connection on $M$ with Christoffel symbols $\Gamma_{ij}^k$ for $i,j,k=1,\hdots,d$.
If $\gamma$ is a geodesic on $M$ with respect to $\nabla$ and $\dot{\gamma}=\sum_{i=1}^d\dot{\gamma}_i\partial_i$, then $\gamma$ satisfies the following differential equation,
\begin{equ}[eq:geodesic-diff-form]
\sum_{k=1}^d \left(\sum_{i,j=1}^d \dot{\gamma}_i\dot{\gamma}_j\Gamma_{ij}^k + \ddot{\gamma}_k \right)\partial_k = 0.
\end{equ}
\label{prop:geodesic-christoffel}
\end{proposition}

\begin{proof}
Let us start writing the condition $\nabla_{\dot{\gamma}}\dot{\gamma}=0$ in terms of Christoffel symbols.
\begin{align}
0
	=&\nabla_{\dot{\gamma}}\dot{\gamma}\label{eq:geodesic-christoffel-proof-0}\\
	\stackrel{\eqref{eq:covariant-linear}}{=}& \sum_{i=1}^d \nabla_{\dot{\gamma}} \dot{\gamma}_i \partial_i\nonumber\\
	\stackrel{\eqref{eq:covariant-leibniz}}{=}& \sum_{i=1}^d \ddot{\gamma}_i \partial_i + \dot{\gamma}_i\nabla_{\dot{\gamma}}\partial_i\nonumber\\
	\stackrel{\eqref{eq:affine-1}}{=}& \sum_{i,j=1}^d \dot{\gamma}_i\dot{\gamma}_j \nabla_{\partial_j}\partial_i + \sum_{i=1}^d\ddot{\gamma}_i \partial_i\nonumber\\
	\stackrel{\eqref{eq:christoffel-symbol}}{=}& \sum_{i,j,k=1}^d \dot{\gamma}_i\dot{\gamma}_j \Gamma_{ij}^k\partial_k + \sum_{i=1}^d\ddot{\gamma}_i \partial_i\nonumber
\end{align}
Consequently,~\eqref{eq:geodesic-christoffel-proof-0} reduces to
\begin{equ}
0 = \sum_{i,j,k=1}^d \dot{\gamma}_i\dot{\gamma}_j \Gamma_{ij}^k \partial_k + \sum_{i=1}^d \ddot{\gamma}_i \partial_i.
\end{equ}
Rearranging this equation, we see that $\gamma$ satisfies~\eqref{eq:geodesic-diff-form}.
\end{proof}

\subsection{Geodesics as length minimizing curves on a Riemannian manifold}
\noindent
Since the metric tensor allows us to measure distances on a Riemannian manifold,  there is an alternative, and sometimes useful, way of defining geodesics on it: as length minimizing curves.
Before we can define a geodesic in this manner, we need to define  the length of a curve on a Riemannian manifold.
This gives rise to a notion of distance between two points as the  minimum length of a  curve that joins these points.
Using the metric tensor we can measure the instantaneous length of a given curve.
Integrating along the vector field induced by its derivative, we can measure the length of the curve.

\begin{definition}[\bf Length  of a curve on a Riemannian manifold]
Let $(M, g)$ be a Riemannian manifold and $\gamma:I\to\R$ be a smooth curve on $M$ where $I\subseteq \R$ is an interval.
We define the length function $L$ as follows:
\begin{equ}[eq:length]
L[\gamma] := \int_I \sqrt{g_{\gamma}\inparen{\dot{\gamma}, \dot{\gamma}}}dt.
\end{equ}
\label{def:line-length}

\end{definition}
 
\noindent
If we ignore the local metric, then this is simply the line integral used for measuring the length of curves in the Euclidean spaces.
On the other hand, the local metric on Euclidean spaces is simply the $2$-norm.
Thus, this definition is a natural generalization of line integrals to Riemannian manifolds.
A notion related to the length of a curve, but mathematically more convenient is in terms of the ``work'' done by a particle moving along the curve.
To define this work, we first need to define the energy function as
\begin{equ}
\mathcal{E}(\gamma,\dot{\gamma},t):=g_{\gamma(t)}(\dot{\gamma}(t), \dot{\gamma}(t)),
\end{equ} 
which is just the square of the term in the integral above.
Given the energy at each time, the work done is the integral of the energy over the curve: 
\begin{equ}[eq:energy]
S[\gamma] = \int_0^1 \mathcal{E}(\gamma, \dot{\gamma}, t)dt.
\end{equ}

\begin{definition}[\bf Geodesics on a Riemannian manifold]
Let $(M,g)$ be a Riemannian manifold and $p,q\in M$.
Let $\Gamma$ be the set of all smooth curves joining $p$ to $q$ defined from $I$ to $M$ where $I\subseteq \R$ is an interval.
$\gamma'\in\Gamma$ is a geodesic on $M$ if $\gamma'$ is a minimizer of
\begin{equ}
\inf_{\gamma\in\Gamma} L[\gamma]
\end{equ}
where $L$ is the length function defined as~\eqref{eq:length}.
\label{def:riemannian-geodesic}
Equivalently, $\gamma'$ minimizes
$S[\gamma]$. 
\end{definition} 

\noindent
The definition of geodesics on Riemannian manifold above  is a solution to an optimization problem and not very useful as such to compute geodesics.
We can recover the dynamics of this curve by using a result from the calculus of variations, the  Euler-Lagrange equations to characterize optimal solutions of this optimization problem.
\begin{theorem}[\bf Euler-Lagrange equations for geodesics]
Let $(M,g)$ be a Riemannian manifold.
Let $p$ and $q$ be two points on $M$, $\gamma':[0,1]\to M$ be the geodesic that joins $p$ to $q$ and minimizes $S[\gamma]$.
Then $\gamma'$ satisfies the following differential equations:
\begin{equ}
\frac{d}{dt} \frac{\partial \mathcal{E}}{\partial \dot{\gamma}}(\gamma') - \frac{\partial \mathcal{E}}{\partial \gamma}(\gamma') = 0.
\end{equ}
\label{thm:geodesic-EL}
\end{theorem}

\noindent
We derive the Euler-Lagrange equations and prove this theorem in~\cref{sec:EL}.

\subsection{Equivalence of the two notions for Riemannian Manifolds}

\noindent
We presented two different definitions of geodesics, one for smooth manifolds with a given affine connection and one for Riemannian manifolds in terms of the metric.
As shown earlier, the metric of a Riemannian manifold determines a unique affine connection -- the Levi-Civita connection.
We also show  that these two formulations  describe the same curves.

\begin{theorem}
Let $(M, g)$ be a $d$-dimensional Riemannian manifold and $\nabla$ be its Levi-Civita connection.
Let $\gamma:I\to M$ be a smooth curve joining $p$ to $q$.
If $\gamma$ is a minimizer of the work function defined in~\eqref{eq:energy}, then
\begin{equ}
\nabla_{\dot{\gamma}}\dot{\gamma} = 0.
\end{equ}
\label{thm:geodesic-equivalence}
\end{theorem}

\noindent
We present the proof of this theorem in~\cref{sec:geodesic-equivalence}.
An immediate corollary of~\cref{thm:geodesic-EL} and~\cref{thm:geodesic-equivalence} is that length minimizing curves are also geodesics in the sense of differentiable manifolds.
The other direction is not necessarily true.
In a given Riemannian manifold, there can be geodesics whose length greater than the distance between points they join.

\subsection{Examples}
\begin{example}[{\bf The Euclidean space}, $\R^n$] 
We do a sanity check and show  that geodesics in $\R^n$ are indeed straight lines.
We have already shown that the Christoffel symbols are $0$;  see~\cref{ex:Rn-christoffel}.
Hence,~\eqref{eq:geodesic-diff-form} simplifies to
\begin{equ}
\sum_{i=1}^n\ddot{\gamma}_i\partial_i = 0.
\end{equ}
Since, $\{\partial_i\}_{i=1}^n$ are independent for each $i=1,\hdots,n$,
\begin{equ}
\ddot{\gamma}_i = 0.
\end{equ}
This is simply the equation of a line.
Therefore, geodesics in the $R^n$ are straight lines.
\end{example}

\begin{figure}[!htb]
\centering
\begin{subfigure}{.4\textwidth}
\centering
\includegraphics[keepaspectratio, width=\linewidth]{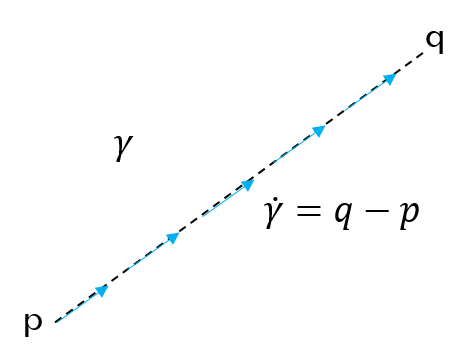}
\caption{The tangent vector to geodesic is parallel to the geodesic at any points.}
\end{subfigure} \hspace{0.5cm}
\begin{subfigure}{.4\textwidth}
\centering
\includegraphics[keepaspectratio, width=\linewidth]{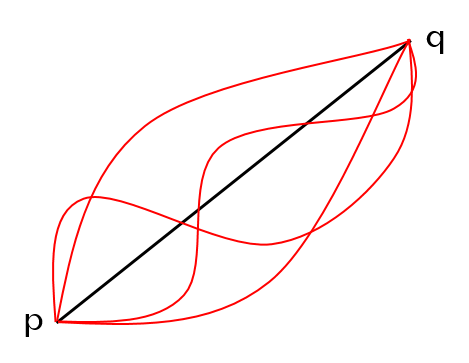}
\caption{Minimizers of length function are geodesics.}
\end{subfigure}
\caption{Geodesics on $\R^2$}
\label{fig:R2-geodesics}
\end{figure}

\begin{figure}[!htb]
\centering
\begin{subfigure}{.4\textwidth}
\centering
\includegraphics[keepaspectratio, width=\linewidth]{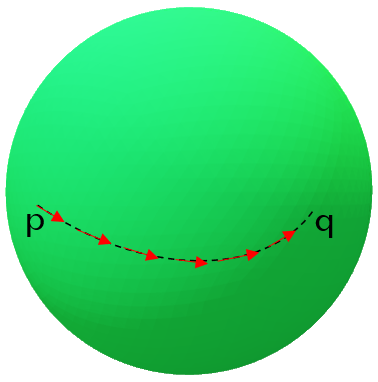}
\caption{The tangent vector to geodesic is parallel to the geodesic at any points.}
\end{subfigure}\hspace{0.5cm}
\begin{subfigure}{.4\textwidth}
\centering
\includegraphics[keepaspectratio, width=\linewidth]{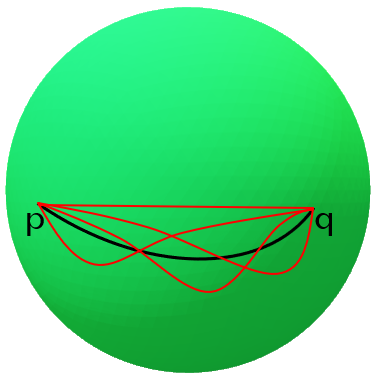}
\caption{Minimizers of length function are geodesics.}
\end{subfigure}
\caption{Geodesics on a sphere}
\label{fig:S2-geodesics}
\end{figure}

\begin{example}[{\bf The positive orthant}, $\R^n_{>0}$]
We show that the geodesic that joins $p$ to $q$ in $\R^n_{>0}$ endowed with metric corresponding to the Hessian of the log-barrier function can be parameterized as follows,
\begin{equ}
\gamma(t) = (p_1(q_1/p_1)^t,\hdots,p_n(q_n/p_n)^t).
\end{equ}
We know that the Christoffel symbols are $\Gamma_{ii}^i(p)=-p_i^{-1}$, and $\Gamma_{ij}^k(p)=0$ for $p\in\R^n_{>0}$ and $i,j,k=1,\hdots,n$ such that $i,j,$ and $k$ are not all same,~\cref{ex:po-christoffel}.
Hence,~\eqref{eq:geodesic-diff-form} simplifies to
\begin{equ}
0	= -\sum_{i=1}^n (-\gamma_i^{-1}\dot{\gamma}_i^2 +  \ddot{\gamma}_i) \partial_i
\end{equ}
Thus,
\begin{equ}
0=\ddot{\gamma}_i - \dot{\gamma}_i^2\gamma_i^{-1}
\end{equ}
or equivalently
\begin{equ}
\frac{d}{dt} \log(\dot{\gamma}_i) = \frac{d}{dt} \log(\gamma_i).
\end{equ}
Consequently,
\begin{equ}
\log(\dot{\gamma}_i) = \log(\gamma_i) + c_i
\end{equ}
for some constant $c_i$.
Equivalently
\begin{equ}
\frac{d}{dt} \log(\gamma_i) = e^{c_i} 
\end{equ}
Therefore,
\begin{equ}
\log(\gamma_i) = \alpha_i t + \beta_i
\end{equ}
for some constants $\alpha_i$ and $\beta_i$.
If we take $\beta_i$ as $\log p_i$ and $(\alpha_i - \beta_i)$ as $\log q_i$, then $\gamma$ becomes,
\begin{equ}
\gamma(t) = (p_1(q_1/p_1)^t,\hdots,p_n(q_n/p_n)^t).
\end{equ}
\cref{fig:geodesic} visualizes a geodesic on $\R^2_{>0}$ and $\R^2$.
\end{example}

\begin{figure}[!hbt]
\centering
\includegraphics[keepaspectratio, width=0.4\textwidth]{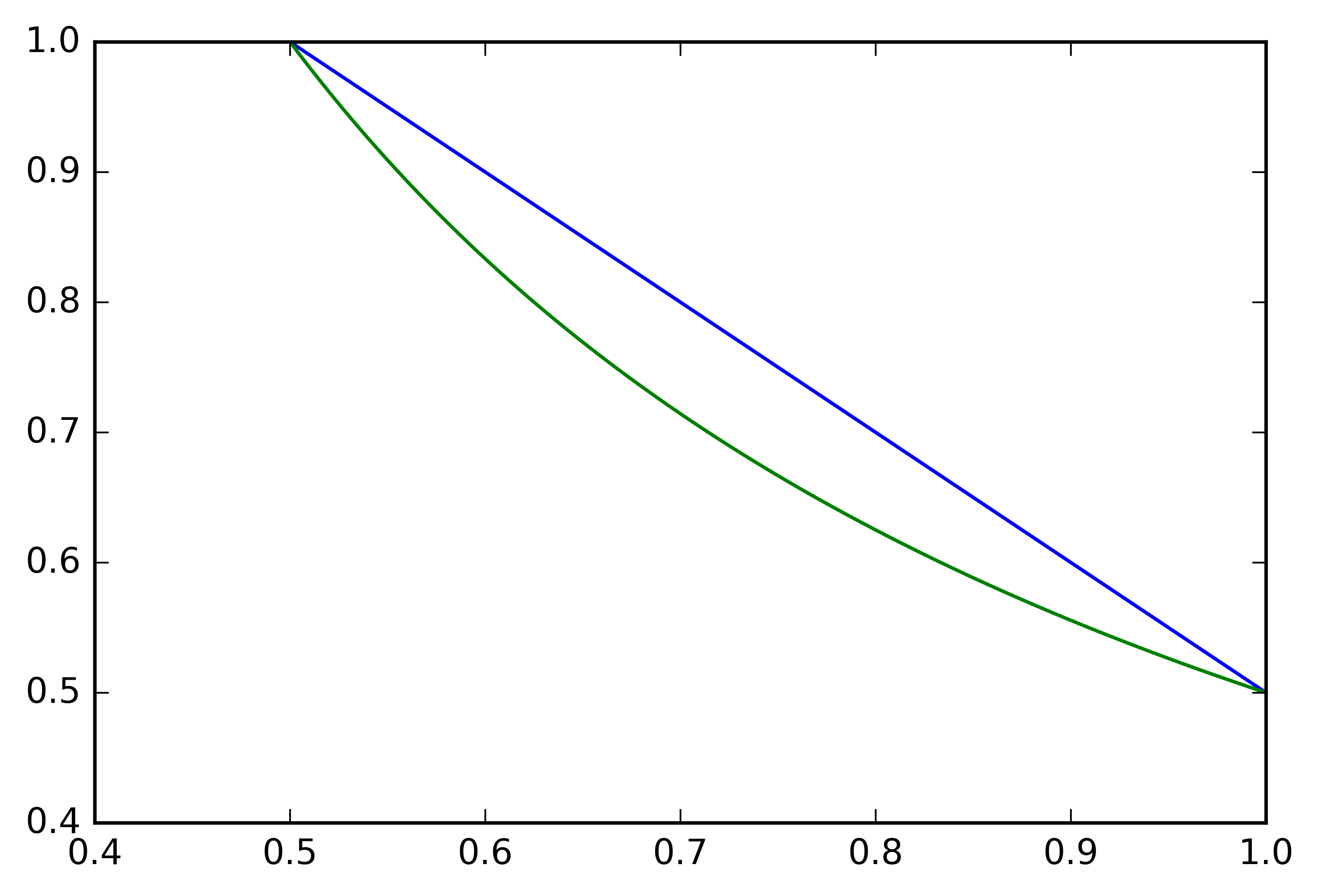}
\caption{
Two geodesics joining $(1.0, 0.5)$ to $(0.5, 1.0)$ on the smooth manifold $\R_{>0}^2$.
The blue curve corresponds to the geodesic with respect to metric tensor $g_p(u,v):=\inner{u, v}$.
The green curve corresponds to the geodesic with respect to metric tensor $g_p(u, v):=\inner{P^{-1} u, P^{-1} v}$ where $P$ is a diagonal matrix whose entries is $p$.
}
\label{fig:geodesic}
\end{figure}

\begin{example}[{\bf The positive definite cone}, $\Sn^n_{++}$]
We show that the geodesic with respect to the Hessian of the logdet metric that joins $P$ to $Q$ on $\Sn^n_{++}$ can be parameterized as follows:
\begin{equ}[eq:pd-geodesics]
\gamma(t) := P^{1/2} (P^{-1/2} Q P^{-1/2})^t P^{1/2}.
\end{equ}
Thus, $\gamma(0)=P$ and $\gamma(1)=Q$.
Instead of using the straight-line formulation, here it turns out to be a bit easier to use the Euler-Lagrange characterization of geodesics to derive a formula.
Let $\{e_i\}_{i=1}^n$ be standard basis for $\R^n$ and 
\begin{equ}
E_{ij}:=
\begin{cases}
 e_i e_j^\top + e_j e_i^\top, &\mbox{ if $i\neq j$},\\
 e_ie_i^\top,&\mbox{ if $i=j$}.
 \end{cases}
 \end{equ}
 $\{E_{ij}\}_{1\leq i\leq j\leq n}$ is a basis for the set of symmetric matrices $\Sn^n$, the tangent space.
 Let $\gamma(t)$ be a minimum length curve that joins $P$ to $Q$.
If we write $\gamma(t)$ as 
 \begin{equ}
 \sum_{1\leq i\leq j\leq n} \gamma_{ij}(t)E_{ij}
 \end{equ}
 then 
 \begin{equ}
 \frac{\partial \gamma(t)}{\partial \gamma_{ij}(t)} = E_{ij}
 \end{equ}
 for $1\leq i\leq j\leq n$.
 Let us define $\mathcal{E}(t)$ as follows,
 \begin{equ}
\mathcal{E}(t) := \frac{1}{2} \tr[\gamma(t)^{-1}\dot{\gamma}(t)\gamma(t)^{-1}\dot{\gamma}(t)]
 \end{equ}
 and notice that this definition of $\mathcal{E}(t)$ is actually the same as $\mathcal{E}(t)$ defined in~\cref{thm:geodesic-EL}.
Let us compute partial derivatives of $\mathcal{E}$, before applying~\cref{thm:geodesic-EL}.
 \begin{align*}
 \frac{\partial \mathcal{E}}{\partial \gamma_{ij}} 
 	=& \frac{1}{2}\tr\left[
 		\frac{\partial \gamma^{-1}}{\partial \gamma_{ij}} \dot{\gamma} \gamma^{-1} \dot{\gamma}
 		+\gamma^{-1} \frac{\partial \dot{\gamma}}{\partial \gamma_{ij}}\gamma^{-1} \dot{\gamma}
 		+\gamma^{-1} \dot{\gamma}\frac{\partial \gamma^{-1}}{\partial \gamma_{ij}} \dot{\gamma}
 		+\gamma^{-1} \dot{\gamma}\gamma^{-1} \frac{\partial \dot{\gamma}}{\partial \gamma_{ij}}
 	\right]\\
 	=& -\tr[\gamma^{-1} E_{ij} \gamma^{-1}\dot{\gamma}\gamma^{-1}\dot{\gamma}]\\
\frac{\partial \mathcal{E}}{\partial \dot{\gamma}_{ij}}
 	=& \frac{1}{2}\tr\left[
 		\frac{\partial \gamma^{-1}}{\partial \dot{\gamma}_{ij}} \dot{\gamma} \gamma^{-1} \dot{\gamma}
 		+\gamma^{-1} \frac{\partial \dot{\gamma}}{\partial \dot{\gamma}_{ij}}\gamma^{-1} \dot{\gamma}
 		+\gamma^{-1} \dot{\gamma}\frac{\partial \gamma^{-1}}{\partial \dot{\gamma}_{ij}} \dot{\gamma}
 		+\gamma^{-1} \dot{\gamma}\gamma^{-1} \frac{\partial \dot{\gamma}}{\partial \dot{\gamma}_{ij}}
 	\right]\\
 	=& \tr[\gamma^{-1}E_{ij}\gamma^{-1}\dot{\gamma}]\\
\frac{d}{dt}\frac{\partial \mathcal{E}}{\partial \dot{\gamma}_{ij}}
	=& \frac{d}{dt}\tr[\gamma^{-1}E_{ij}\gamma^{-1}\dot{\gamma}]\\
	=& \tr[-\gamma^{-1}\dot{\gamma}\gamma^{-1}E_{ij}\gamma^{-1}\dot{\gamma}-\gamma^{-1}E_{ij}\gamma^{-1}\dot{\gamma}\gamma^{-1}\dot{\gamma}+\gamma^{-1}E_{ij}\gamma^{-1}\ddot{\gamma}]\\
	=&\tr[E_{ij}(\gamma^{-1}\ddot{\gamma}\gamma^{-1}-2\gamma^{-1}\dot{\gamma}\gamma^{-1}\dot{\gamma}\gamma^{-1})]
 \end{align*}
for $1\leq i\leq j\leq n$.
Thus,
\begin{equ}
\tr[E_{ij}(\gamma^{-1}\ddot{\gamma}\gamma^{-1}-2\gamma^{-1}\dot{\gamma}\gamma^{-1}\dot{\gamma}\gamma^{-1})] = -\tr[\gamma^{-1} E_{ij} \gamma^{-1}\dot{\gamma}\gamma^{-1}\dot{\gamma}],
\end{equ}
or equivalently
\begin{equ}
\inner*{\gamma^{-1}\ddot{\gamma}\gamma^{-1}-\gamma^{-1}\dot{\gamma}\gamma^{-1}\dot{\gamma}\gamma^{-1},E_{ij}}_F = 0
\end{equ}
for $1\leq i\leq j\leq n$.
Since $E_{ij}$ forms a basis for $\Sn^n$, this implies that
\begin{equ}
\gamma^{-1}\ddot{\gamma}\gamma^{-1}-\gamma^{-1}\dot{\gamma}\gamma^{-1}\dot{\gamma}\gamma^{-1} = 0,
\end{equ}
or equivalently
\begin{equ}
0 = \ddot{\gamma}\gamma^{-1}-\dot{\gamma}\gamma^{-1}\dot{\gamma}\gamma^{-1} =\frac{d}{dt}\left(\dot{\gamma}\gamma^{-1}\right).
\end{equ}
Hence,
\begin{equ}
\dot{\gamma}\gamma^{-1} = C
\end{equ}
for some constant real matrix $C$. 
Equivalently,
\begin{equ}[eq:pd-proof-0]
\dot{\gamma}=C\gamma.
\end{equ}
If $\gamma$ was a real valued function and $C$ was a real constant, then~\eqref{eq:pd-proof-0} would be a first order linear differential equation whose solution is $\gamma(t)=d\exp(ct)$ for some constant $d$.
Using this intuition our guess for $\gamma(t)$ is $\exp(tC)D$ for some constant matrix $D$.
We can see that this is a solution by plugging into~\eqref{eq:pd-proof-0}, but we need to show that this is the unique solution to this differential equations.
Let us define $\phi(t):=\exp(-tC) \gamma(t)$ and compute $\dot{\phi}$,
\begin{align*}
\dot{\phi}(t)
	=& -C\exp(-tC)\gamma(t) + \exp(-tC)\dot{\gamma}(t)
	=& -\exp(-tC)C\gamma(t)+\exp(-tC)C\gamma(t)
	=0.
\end{align*}
Here, we used the fact that $\exp(tC)$ and $C$ commutes.
Thus $\phi(t) = D$ for some constant matrix.
Hence, the unique solution of~\eqref{eq:pd-proof-0} is $\gamma(t)=\exp(tC)D$.
Now, let us notice that both $\gamma$ and $\dot{\gamma}$ are symmetric.
In other words,
\begin{equ}
\dot{\gamma}= \exp(tC)CD=C\exp(tC)D=C\gamma(t)=\gamma(t)C^\top=\exp(tC)DC^\top.
\end{equ}
In other words, $CD=DC^\top$.
Now, let us consider boundary conditions,
\begin{equ}
P =\gamma(0) = \exp(0\cdot C)D = D,
\end{equ}
and write $C$ as $P^{1/2} S P^{-1/2}$.
Such a matrix $S$ exists as $P$ is invertible.
Thus, the $CD=DC^\top$ condition implies
\begin{equ}
P^{1/2} S P^{1/2} = P^{1/2} S^\top P^{1/2}
\end{equ} 
or equivalently $S=S^\top$.
Thus, $C=P^{1/2} S P^{-1/2}$ for some symmetric matrix, $S$.
Hence,
\begin{equ}
\gamma(t) = \exp(tC)D=\left(\sum_{k=0}^\infty \frac{t^k (P^{1/2} S P^{-1/2})^k}{k!}\right)P = P^{1/2} \left(\sum_{k=0}^\infty \frac{t^k S^k}{k!}\right)P^{1/2} = P^{1/2}\exp(tS) P^{1/2}.
\end{equ}
Finally, considering $\gamma(1)=Q$, we get
\begin{equ}
S = \log(P^{-1/2} Q P^{-1/2}).
\end{equ}
Therefore, the geodesic joining $P$ to $Q$ is
\begin{equ}
\gamma(t) = P^{1/2} (P^{-1/2}QP^{-1/2})^t P^{1/2}.
\end{equ}

\label{ex:pd-cone-geodesics}
\end{example}

\section{Geodesic Convexity}
\label{sec:g-convexity}
With all the machinery and the definitions of geodesics behind us, we are now ready to define geodesic convexity. 
We refer the reader to the book \cite{Udr94} for an extensive treatment on geodesic convexity.

\subsection{Totally convex sets}
\begin{definition}[\bf Totally (geodesically) convex set]
Let $(M, g)$ be a Riemannian manifold.
A set $K\subseteq M$ is said to be  totally  convex with respect to $g$, if for any $p,q\in K$, any geodesic $\gamma_{pq}$ that joins $p$ to $q$ lies entirely in $K$. 
\label{def:g-convex-set}
\end{definition}

\noindent
Totally  convex sets are a generalization of convex sets.
In the Euclidean case, there is a unique geodesic joining points $p$ and $q$ which is the straight line segment between $p$ and $q$.
Consequently, totally  convex sets and convex sets are same in the Euclidean case.
One can relax the definition of totally  convex sets by requiring that geodesics that minimize the distance between points.
These sets are called \emph{geodesically convex sets}.
If there is a unique geodesic  joining each pair of points, then both the definition are the same.
However, in general totally  convex sets are more restrictive.
This can be seen from the following example of the unit sphere.
On  $S^n:=\Set{x\in\R^{n+1}}{\norm{x}_2=1}$ with the metric induced by the Euclidean norm, there are at least two geodesics joining given any two points, $p,q\in S^n$.
These geodesics are long and short arcs between $p$ and $q$ on a great circle passing through both $p$ and $q$.
A set on $S^n$ is geodesically convex if short arcs are part of the set.
On the other hand, a set on $S^n$ is totally  convex if both short and long arcs are part of the set.

\begin{example}[\bf A non-convex but totally  convex set]
Let us consider the Riemannian manifold on $\Sn^n_{++}$ and a positive number $c\in\R+$.
Let $$D_c :=\Set{P\in\Sn^n_{++}}{\det(P)=c}.$$
One can easily verify that $D_c$ is a non-convex set.
On the other hand, if $P$ and $Q$ are two points in $D_c$, then the geodesic joining $P$ to $Q$, $\gamma_{PQ}$ can be parameterized as follows
\begin{equ}
\gamma_{PQ}(t):= P^{1/2} (P^{-1/2}QP^{-1/2})^t P^{1/2},
\end{equ}
see~\cref{ex:pd-cone-geodesics}.
Now, let us verify that $\forall t\in[0,1]$, $\gamma_{PQ}(t)\in D_c$ or equivalently $\det(\gamma_{PQ}(t))=c$.
\begin{align*}
\det(\gamma_{PQ}(t))
	=& \det(P^{1/2} (P^{-1/2}QP^{-1/2})^t P^{1/2})\\
	=& \det(P)^{1/2} (\det(P)^{-1/2}) \det(Q) \det(P)^{-1/2})^c \det(P)^{1/2}\\
	=& \det(P)^{1-t}\det(Q)^t\\
	=& c^{1-t}c^t=c.
\end{align*}
Therefore, $D_c$ is a totally  convex but non-convex subset of $\Sn^n_{++}$.

\end{example}

\begin{figure}[!htb]
\centering
\includegraphics[keepaspectratio, width=0.7\textwidth]{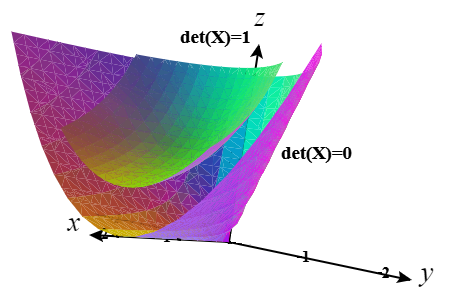}
\caption[]{Visual representation of a geodesically convex set on  $\Sn^2_{++}$.
Given a $2 \times 2$ matrix $X=\bigl( \begin{smallmatrix}x & y\\ y & z\end{smallmatrix}\bigr)$, 
the outer surface, $\det(X)=0$, is the boundary of $\Sn^2_{++}$ and  the inner surface, $\det(X)=1$, is a geodesically convex set consisting of solutions of $\det(X)=1$.
}
\label{fig:g-convex-set}
\end{figure}

\subsection{Geodesically convex functions}
\begin{definition}[\bf Geodesically convex function]
Let $(M, g)$ be a Riemannian manifold and $K\subseteq M$ be a totally convex set with respect to $g$.
A function $f:K\to\R$ is said to be a geodesically convex function with respect to $g$ if for any $p,q\in K$, and for any geodesic $\gamma_{pq}:[0,1]\to K$ that joins $p$ to $q$,
\begin{equ}
\forall t\in[0,1]\;\;f(\gamma_{pq}(t))\leq (1-t) f(p) + t f(q).
\end{equ}
\label{def:g-convex-function}
\end{definition}

\noindent
This definition can be interpreted as follows: the univariate function constructed by restricting $f$ to a geodesic is a convex function if the geodesic lies in $K$.

\begin{theorem}[\bf First-order characterization of geodesically convex functions]
Let $(M, g)$ be a Riemannian manifold and $K\subseteq M$ be an open and  totally convex set with respect to $g$.
A differentiable function $f:K\to\R$ is said to be geodesically convex  with respect to $g$ if and only if for any $p,q\in K$, and for any geodesic $\gamma_{pq}:[0,1]\to K$ that joins $p$ to $q$,
\begin{equ}[eq:first-order-condition]
f(p)+\dot{\gamma}_{pq}(f)(p)\leq f(q),
\end{equ}
where $\dot{\gamma}_{pq}(f)$ denotes the first derivative of $f$ along the geodesic.
\label{thm:g-convex-first-order}
\end{theorem}

\noindent
Geometrically, ~\cref{thm:g-convex-first-order} is equivalent to saying that the linear approximation given by the tangent vector at any point is a lower bound for the function.
Recall that convex functions also satisfy this property.
The first derivative of $f$ along geodesics simply corresponds to the directional derivative of $f$ in that case.

\begin{figure}[!htb]
\centering
\includegraphics[keepaspectratio, width=0.6\textwidth]{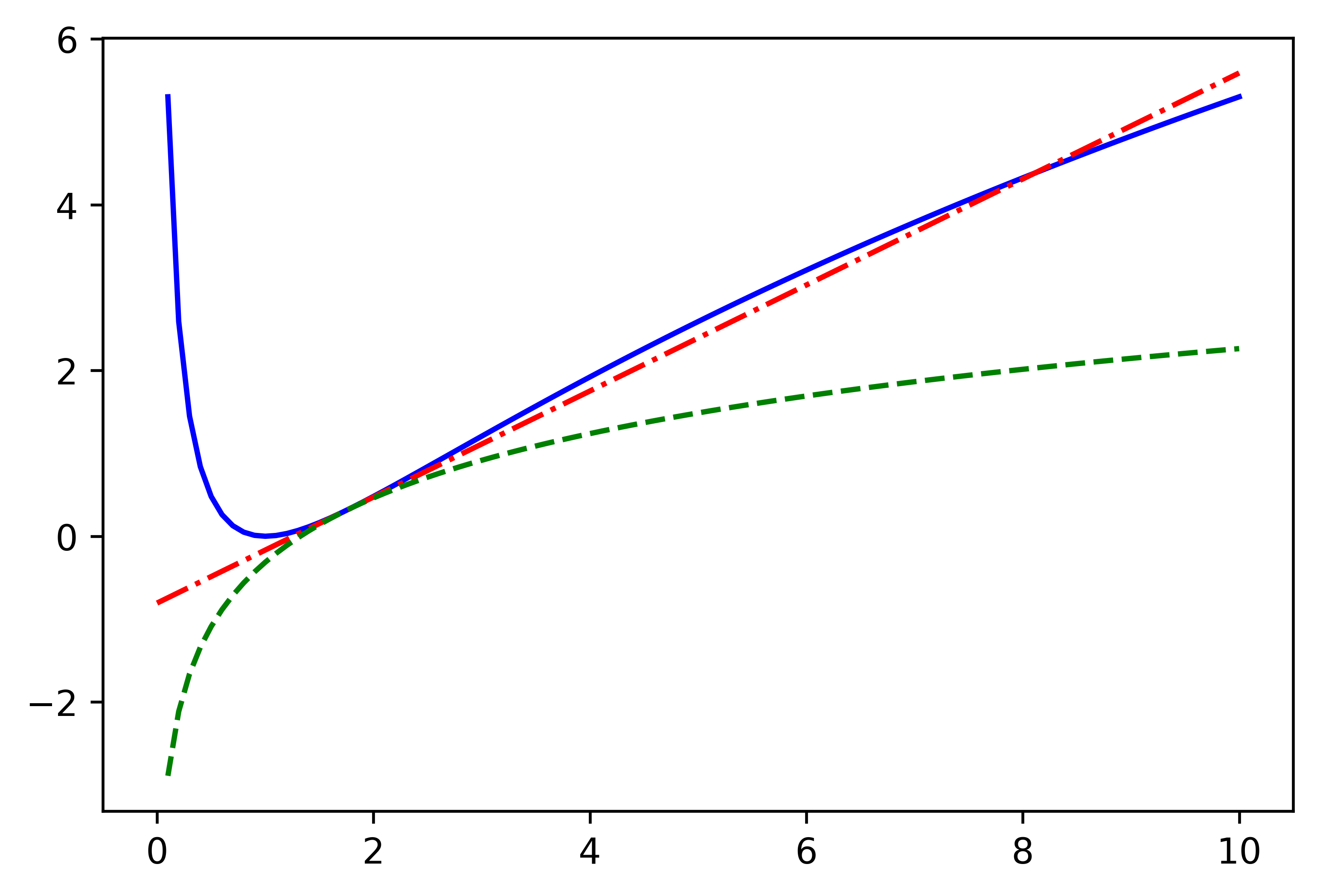}
\caption{%
The blue line is the graph of $f(x):=(\log x)^2$ is a non-convex function.
The red line corresponds to the tangent line at $x=2$.
The green line corresponds to  $f(2)+\dot{\gamma}_{2,x}(f)(2)$ where $\gamma$ is the geodesic with respect to metric $g_x(u,v):=ux^{-1}v$.
}
\label{fig:first-order}
\end{figure}

\begin{proof}
Let us first show that if $f$ is geodesically convex, then~\eqref{eq:first-order-condition} holds.
Geodesic convexity, by definition, implies that for any $p,q\in K$, any geodesic joining $p$ to $q$, $\gamma_{pq}$ with $\gamma_{pq}(0)=p$, $\gamma_{pq}(1)=q$ and $t\in(0,1]$,
\begin{equ}[eq:first-order-proof-0]
f(p) + \frac{f(\gamma_{pq}(t))-f(p)}{t} \leq f(q).
\end{equ}
Recall that
\begin{equ}
\dot{\gamma_{pq}}(f)(p) = \lim_{t\to 0}\frac{f(\gamma_{pq}(t))-f(p)}{t}.
\end{equ}
Hence, if we take limit the of~\eqref{eq:first-order-proof-0}, we get
\begin{equ}
f(p)+\dot{\gamma}_{pq}(f)(p)\leq f(q).
\end{equ}

\noindent
Now, let us assume that for any $p,q\in K$, and for any geodesic $\gamma_{pq}:[0,1]\to K$ that joins $p$ to $q$~\eqref{eq:first-order-condition} holds.
Given $p,q$ and $\gamma_{pq}$, let us fix a $t\in(0,1)$.
Let us denote $\gamma_{pq}(t)$ by $r$.
Next, let us consider curves $\alpha(u) := \gamma_{pq}(t+u(1-t))$ and $\beta(u):=\gamma_{pq}(t-ut)$.
These curves are a reparametrization of $\gamma_{pq}$ from $r$ to $q$ and $r$ to $p$.
Consequently, they are geodesics joining $r$ to $q$ and $r$ to $p$.
Their derivatives with respect to $u$ are 
$$\dot{\alpha}(0)=(1-t)\dot{\gamma_{pq}}(t)$$ and $$\dot{\beta}(0)=-t\dot{\gamma_{pq}}(t).$$
Thus, if we apply~\eqref{eq:first-order-condition} to $r,q$ and $\alpha$, then we get
\begin{equ}[eq:first-order-proof-1]
f(q) \geq \dot{\alpha}(f)(r) + f(r) = f(r) + (1-t)\dot{\gamma_{pq}}(f)(r) .
\end{equ}
Similarly, if we apply~\eqref{eq:first-order-condition} to $r,p$ and $\beta$, then we get
\begin{equ}[eq:first-order-proof-2]
f(p) \geq \dot{\beta}(f)(r) + f(r) = f(r) -t\dot{\gamma_{pq}}(f)(r) .
\end{equ}
If we multiply~\eqref{eq:first-order-proof-1} with $t$ and~\eqref{eq:first-order-proof-2} with $1-t$, and sum them up, then we get
\begin{equ}
tf(q)+(1-t)f(p) \geq f(r) = f(\gamma_{pq}(t)).
\end{equ}
Therefore, $f$ is geodesically convex.
\end{proof}

\begin{theorem}[\bf Second-order characterization of geodesically convex functions]
Let $(M, g)$ be a Riemannian manifold and $K\subseteq M$ be an open and totally convex set with respect to $g$.
A twice differentiable function $f:K\to\R$ is said to be a geodesically convex function with respect to $g$ if and only if for any $p,q\in K$, and for any geodesic $\gamma_{pq}$ that joins $p$ to $q$,
\begin{equ}
\frac{d^2 f(\gamma_{pq}(t))}{dt^2} \geq 0.
\end{equ}
\label{thm:g-convex-second-order}
\end{theorem}

\begin{proof}
Given $p,q\in K$ and $\gamma_{pq}$, let us define $\theta:[0,1]\to\R$ as $\theta(t)=f(\gamma_{pq}(t))$.
If $f$ is geodesically convex, then $\forall t\in[0,1]$,
\begin{equ}
f(\gamma_{pq}(t)) \leq (1-t) f(p) + t f(q)
\end{equ}
or equivalently
\begin{equ}
\theta(t) \leq (1-t) \theta(0) + t \theta(1).
\end{equ}
In other words, $\theta$ is a convex function.
The second order characterization of convex functions leads to
\begin{equ}
0\geq \frac{d^2 \theta(t)}{d t^2} = \frac{d^2 f(\gamma_{pq}(t))}{dt^2}.
\end{equ}
On the other hand, if $f$ is not geodesically convex, then there exist $p, q\in K$, a geodesic $\gamma_{pq}$ joining $p$ to $q$ and $t\in[0,1]$ such that
\begin{equ}
f(\gamma_{pq}(t)) > (1-t) f(p) + t f(q).
\end{equ}
Thus,
\begin{equ}
\theta(t) > (1-t) \theta(0) + t \theta(1).
\end{equ}
Therefore, $\theta$ is not convex and consequently there exists $u\in[0,1]$ such that
\begin{equ}
0>\frac{d^2 \theta(u)}{d t^2} = \frac{d^2 f(\gamma_{pq}(u))}{dt^2}.
\end{equ}
Hence, $f$ is geodesically convex if and only if
\begin{equ}
\frac{d^2 f(\gamma_{pq}(t))}{dt^2} \geq 0.
\end{equ}
\end{proof}

\subsection{Examples of geodesically convex functions}
In this section, we present some simple examples of geodesically convex functions.

\begin{example}[\bf Geodesically convex functions on the positive orthant]
Let us denote the set of multivariate polynomials with positive coefficients over $\R^n$ by $P^n_+$.
Also, let us recall that geodesics on $\R^n_{>0}$ are of the form $\exp(\alpha t+\beta)$ for $\alpha,\beta\in\R^n$ where $\exp$ function is evaluated independently on each coordinate.
\begin{itemize}
	\item 
	The log barrier function, $\inner{1, \log(x)}$ is both geodesically convex and concave, where $1$ denotes the vector whose coordinates are $1$. 
	We can verify this by observing that the restriction of $\inner{1, \log(x)}$ to a geodesic is $\inner{1, \alpha}t+\inner{1, \beta}$, a linear function of $t$.
	Thus, $\inner{1, \log(x)}$ is geodesically convex by the second order characterization of geodesically convex functions.
	
	\item 
	Multivariate polynomials with positive coefficients, $p\in P^n_+$, are geodesically convex.
	Let use denote monomial $\prod_{i=1}^n x^{\lambda_i}$ by $x^\lambda$ for $\lambda\in\Z_{\geq 0}^n$.
	The important observation is that if we  restrict a monomial to  a geodesic, then it has the form $\exp(\inner{\lambda, \alpha}t+\inner{\lambda, \beta})$ and its first derivative is $\inner{\lambda, \alpha}\exp(\inner{\lambda, \alpha}t + \inner{\lambda, \beta})$.
	Thus, the second derivative of any monomial is positive along any geodesic.
	Consequently, the second derivative of $p$ is positive along any geodesic since its coefficients are positive.
	Our observation can be also interpreted as when we compute the derivative of a monomial along a geodesic, the  result is a scaling of the same monomial.
	In contrast, the derivative along a straight line results in a polynomial whose degree is $1$ less than the differentiated monomial.  
	
	\item
	$\log(p(x))$ for $p\in P^n_+$.
	Let us fix a geodesic $\gamma(t):=\exp(\alpha t + \beta)$ for $\alpha,\beta\in\R^n$.
	Let us write $p(\gamma(t))$ as follows,
	\begin{equ}
	p(\gamma(t)) := \sum_{\theta\in \mathcal{F}} c_\theta \exp(\inner{\theta,\alpha}t+\inner{\theta, \beta})
	\end{equ}
	where $\mathcal{F}\subseteq\Z_{\geq 0}^n$ and $c_\theta$ is the coefficient of the monomial $\prod_{i=1}^n x_i^{\theta_i}$.
	Now, let us compute the Hessian of $\log(p(\gamma(t)))$,
	\begin{align*}
		\frac{d \log(p(\gamma(t)))}{dt} =& \frac{\sum_{\theta\in\mathcal{F}} c_\theta \inner{\theta,\alpha} \exp(\inner{\theta,\alpha}t+\inner{\theta,\beta})}{\sum_{\theta\in\mathcal{F}} c_\theta \exp(\inner{\theta,\alpha}t+\inner{\theta,\beta})}\\
		\frac{d^2 \log(p(\gamma(t)))}{dt^2} =& \frac{\sum_{\theta\in\mathcal{F}} c_\theta \inner{\theta,\alpha}^2 \exp(\inner{\theta,\alpha}t+\inner{\theta,\beta})}{\sum_{\theta\in\mathcal{F}} c_\theta \exp(\inner{\theta,\alpha}t+\inner{\theta,\beta})} \\ 
		&-\left(\frac{\sum_{\theta\in\mathcal{F}} c_\theta \inner{\theta,\alpha} \exp(\inner{\theta,\alpha}t+\inner{\theta,\beta})}{\sum_{\theta\in\mathcal{F}} c_\theta \exp(\inner{\theta,\alpha}t+\inner{\theta,\beta})}\right)^2\\
		=&\frac{\sum_{\theta,\theta'\in\mathcal{F}} 
		(c_\theta \inner{\theta,\alpha}-c_{\theta'}\inner{\theta',\alpha})^2 
		\exp(\inner{\theta,\alpha}t+\inner{\theta,\beta})\exp(\inner{\theta',\alpha}t+\inner{\theta',\beta})
		}
		{\left(\sum_{\theta\in\mathcal{F}} c_\theta \exp(\inner{\theta,\alpha}t+\inner{\theta,\beta})\right)^2}.
	\end{align*}
	Non-negativity of $\frac{d^2 \log(p(\gamma(t)))}{dt^2}$ follows from the non-negativity of $\exp(x)$.
	Thus, $\log(p(x))$ is geodesically convex by the second order characterization of geodesically convex functions.
\end{itemize}
\end{example}

\begin{proposition}[\bf Geodesic linearity of logdet]
$\log\det(X)$  is geodesically both convex and concave on $\Sn^n_{++}$ with respect to the metric $g_X(U, V):=\tr[X^{-1} U X^{-1} V]$.
\label{prop:log-det}
\end{proposition}
\begin{proof}
Let $X, Y\in\Sn^n_{++}$ and $t\in[0,1]$. 
Then, the geodesic joining $X$ to $Y$ is 
$$\gamma(t)=X^{1/2}(X^{-1/2}YX^{-1/2})^t X^{1/2},$$ see \cref{ex:pd-cone-geodesics}.
Thus, 
\begin{equ}
\log\det(\gamma(t)) = \log\det(X^{1/2} (X^{-1/2}Y X^{-1/2})^t X^{1/2}) = (1-t)\log\det(X) + t\log\det(Y).
\end{equ}
Therefore, $\log\det(X)$ is a geodesically linear function over the positive definite cone with respect to the metric $g$.
\end{proof}

\begin{proposition}
Let $T(X)$ be a strictly positive linear operator from $\Sn^n_{++}$ to $\Sn^m_{++}$, in other words it maps positive definite matrices to positive definite matrices.
$T(X)$ is geodesically convex with respect to Loewner partial ordering on $\Sn^m_{++}$ over $\Sn^n_{++}$ with respect to metric $g_X(U, V):=\tr[X^{-1} U X^{-1} V]$.
In other words, for any geodesic, $\gamma:[0,1]\to\Sn^n_{++}$,
\begin{equ}[eq:pd-linear-convexity]
\forall t\in[0,1],\;\;T(\gamma(t))\preceq (1-t)T(\gamma(0)) + t T(\gamma(1)).
\end{equ}
\end{proposition}
\begin{proof}
Any linear operator on $\Sn^n_{++}$ can be written as $T(X):=\sum_{i=1}^d A_i X B_i$ for some $m\times n$ matrices $A_i$ and $n\times m$ matrices $B_i$.
If we can show that $\frac{d^2}{dt^2} T(\gamma(t))$ is positive definite, then this implies~\eqref{eq:pd-linear-convexity}.
Let us consider the geodesic $\gamma(t):=P^{1/2}\exp(tQ)P^{1/2}$ for $P\in\Sn^n_{++}$ and $Q\in\Sn^n$.
The second derivative of $T$ along $\gamma$ is
\begin{align*}
\frac{d T(\gamma(t))}{d t} =& \sum_{i=1}^d A_i P^{1/2} Q\exp(tQ)P^{1/2} B_i\\
\frac{d^2 T(\gamma(t))}{d t^2} =& \sum_{i=1}^d A_i P^{1/2} Q\exp(tQ)QP^{1/2} B_i\\
					=& T(P^{1/2} Q\exp(tQ)QP^{1/2}).
\end{align*}
We can see that $P^{1/2} Q\exp(tQ)QP^{1/2}$ is positive definite for $t\in[0,1]$ and $$T(P^{1/2} Q\exp(tQ)QP^{1/2})$$ is also positive definite as $T$ is a strictly positive linear map.
Consequently, $\frac{d^2}{dt^2} T(\gamma(t))$ is positive definite, and~\eqref{eq:pd-linear-convexity} holds.
\end{proof}

\begin{proposition}[\bf Geodesic convexity of logdet of  positive operators]
Let $T(X)$ be a strictly positive linear operator.
Then, 
$\log\det(T(X))$ is geodesically convex on $\Sn^n_{++}$ with respect to the metric $g_X(U, V):=\tr[X^{-1} U X^{-1} V]$.
\label{prop:log-det-2}
\end{proposition}
\begin{proof}
Let us write $T(X)$ as 
$$T(X):=\sum_{i=1}^d A_i X B_i$$ for some $m\times n$ matrices $A_i$ and $n\times m$ matrices $B_i$.
We need to show that the Hessian of $\log\det(T(X))$ is positive semi-definite along any geodesic.
Let us consider the geodesic 
$$\gamma(t):=P^{1/2}\exp(tQ)P^{1/2}$$  for $P\in\Sn^n_{++}$ and $Q\in\Sn$.
The second derivative of $\log\det(T(X))$ along $\gamma$ is
\begin{align*}
\frac{d \log\det(T(\gamma(t)))}{d t} 
	=& \tr\left[T(\gamma(t))^{-1}\frac{d}{dt} T(\gamma(t))\right]\\					
\frac{d^2 \log\det(T(\gamma(t)))}{d t^2}
	=& \tr\left[-T(\gamma(t))^{-1}\frac{d}{dt} T(\gamma(t))T(\gamma(t))^{-1}\frac{d}{dt} T(\gamma(t)) + T(\gamma(t))^{-1}\frac{d^2}{dt^2}T(\gamma(t))\right].
\end{align*}
We need to only verify that 
$$\left.\frac{d^2 \log\det(T(\gamma(t)))}{d t^2}\right\rvert_{t=0}\geq0.$$
In other words, we need to show that 
\begin{equ}
\tr\left[T(P)^{-1}\left(T(P^{1/2}Q^2 P^{1/2}) -T(P^{1/2}QP^{1/2})T(P)^{-1} T(P^{1/2}QP^{1/2})\right)\right]\geq 0.
\end{equ}
In particular, if we show that
\begin{equ}[eq:need-to-show]
T(P^{1/2}Q^2 P^{1/2}) \succeq T(P^{1/2}QP^{1/2})T(P)^{-1} T(P^{1/2}QP^{1/2}),
\end{equ}
then we are done.
Let us define another strictly positive linear operator
\begin{equ}
T'(X):=T(P)^{-1/2} T(P^{1/2}XP^{1/2})T(P)^{-1/2}.
\end{equ}
If $T'(X^2)\succeq T'(X)^2$, then by picking $X=Q$ we arrive at~\eqref{eq:need-to-show}.
This inequality is an instance of  Kadison's inequality, see~\cite{bhatia2009positive} for more details.
A classical result from matrix algebra is $A\succeq BD^{-1}C$ if and only if
\begin{equ}
\begin{bmatrix}
A&B\\
C&D
\end{bmatrix}\succeq 0
\end{equ}
where $A$ and $D$ square matrices, $B$ and $C$ are compatible sized matrices and $D$ is invertible.
The quantity is called \emph{Schur complement}, $A-BD^{-1}C$.
Thus, we need to verify that 
\begin{equ}
\begin{bmatrix}
T'(X^2)&T'(X)\\
T'(X)&I_m
\end{bmatrix}\succeq 0.
\end{equ}
Let $\{u_i\}_{i=1}^n$ be eigenvalues of $X$ that form an orthonormal basis and $\{\lambda_i\}$ be the corresponding eigenvalues.
Then, 
\begin{equ}
X=\sum_{i=1}^n \lambda_i u_iu_i^\top,
\end{equ}
and
\begin{equ}
X^2=\sum_{i=1}^n \lambda_i^2 u_i u_i^\top.
\end{equ}
Let us denote $T'(u_iu_i^\top)$ by $U_i$, then 
\begin{equ}
I_m = T'(I_n) = T'(\sum_{i=1}^n u_iu_i^\top) = \sum_{i=1}^n U_i
\end{equ}
and consequently,
\begin{equ}
\begin{bmatrix}
T'(X^2)&T'(X)\\
T'(X)&I_m
\end{bmatrix}
=
\begin{bmatrix}
T'(\sum_{i=1}^n\lambda_i^2 u_iu_i^\top)&T'(\sum_{i=1}^n \lambda_i u_iu_i^\top)\\
T'(\sum_{i=1}^n\lambda_i u_iu_i^\top)&\sum_{i=1}^n U_i
\end{bmatrix}
=
\sum_{i=1}^n
\begin{bmatrix}
\lambda_i^2 U_i& \lambda_i U_i\\
\lambda_i U_i& U_i
\end{bmatrix}.
\end{equ}
Since $\lambda_i^2 U_i - (\lambda_i U_i)U_i^{-1}(\lambda_i U_i) = 0$, 
\begin{equ}
\sum_{i=1}^n
\begin{bmatrix}
\lambda_i^2 U_i& \lambda_i U_i\\
\lambda_i U_i& U_i
\end{bmatrix}\succeq 0.
\end{equ}
Thus, $T'(X^2)\succeq T'(X)$ for any $X\in\Sn^n$.
Therefore,~\eqref{eq:need-to-show} holds and $\log\det(X)$ is a geodesically convex function.
\end{proof}

\subsection{What functions are not geodesically convex?}
A natural question one can ask is that given a manifold $M$ and a smooth function $f:M\to\R$, does there exist a metric $g$ such that $f$ is geodesically convex on $M$ with respect to $g$.
In general, verifying if the function $f$ is geodesically convex or not with respect to the metric $g$ is relatively easy.
 However, arguing that no such metric exists is not so easy.
 We already saw that  non-convex functions can be geodesically convex with respect to  metrics induced by the Hessian of  log-barrier functions.
Functions that are not geodesically convex with respect to these metrics can still be geodesically convex with respect to other metrics.

To prove that something is not geodesically convex for any metric, we revisit convexity.
We know that any local minimum of a convex function is also its global minimum.
This property also extends to geodesically convex functions.
Consequently, one class of functions that are not geodesically convex with respect to any metric are functions that have a local minimum that is  not a global minimum.

\begin{proposition}[\bf Establishing non-geodesic convexity]
Let $M$ be a smooth manifold, and $f$ be a function such that there exists $p\in M$ and an open neighborhood of $p$, $U_p$, such that 
\begin{equ}[eq:local-minima]
f(p) = \inf_{q\in U_p} f(q)
\end{equ}
but
\begin{equ}[eq:not-global-minima]
f(p) > \inf_{q\in M} f(q).
\end{equ}
Then there is no metric tensor $g$ on $M$ such that $f$ is geodesically convex with respect to $g$.
\label{prop:geodesically-non-convex-functions}
\end{proposition}

\noindent
\eqref{eq:local-minima} corresponds to $p$ being a local minimum of $f$, while~\eqref{eq:not-global-minima} corresponds to $p$ not being a global minimum of $p$.

\begin{proof}
Let us assume that there exists a metric $g$ such that $f$ is geodesically convex with respect to $g$.
Let $q\in M$ be such that $f(q)<f(p)$ and $\gamma:[0,1]\to M$ be a geodesic such that $\gamma(0)=p$ and $\gamma(1)=q$.
Since $f$ is geodesically convex, we have
\begin{equ}
\forall t\in[0,1],\;\;f(\gamma(t))\leq (1-t)f(\gamma(0)) + tf(\gamma(1)) = (1-t)f(p)+tf(q)<f(p).
\end{equ}  
For some $t_0\in (0,1]$, $\gamma(t)\in U_p$ for $t\in[0,t_0)$ as $\gamma$ is smooth.
Then, $$f(\gamma(t))\geq f(p) \mbox{  for  } t\in[0, t_0)$$ by~\eqref{eq:local-minima}.
This is a contradiction.
Therefore, there is no metric $g$ on $M$ such that $f$ is geodesically convex with respect to $g$.
\end{proof}

\begin{figure}[!hbt]
\centering
\includegraphics[keepaspectratio, width=0.7\textwidth]{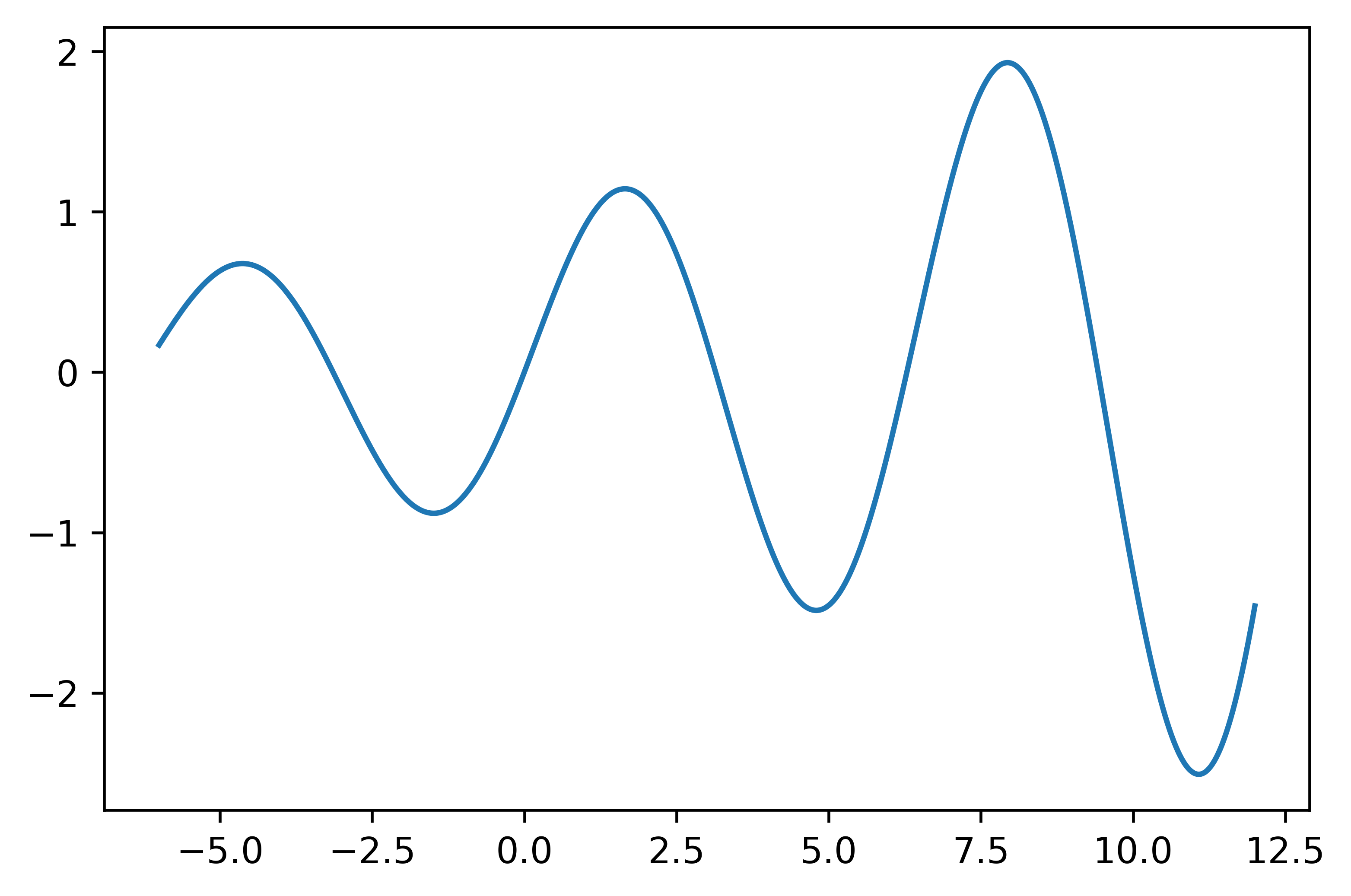}
\caption{%
A function that is not geodesically convex on $\R$ -- $\sin(x)\exp(x/12)$.
}
\label{fig:non-g-convex-function}
\end{figure}

\section{Application: Characterizing Brascamp-Lieb Constants}
\label{sec:BL}
The Brascamp-Lieb inequality~\cite{BL76} is a powerful mathematical tool which unifies most of the well-known classical inequalities with  a single formulation.
The inequality can be described as follows.
\begin{definition}[\bf Brascamp-Lieb inequality~\cite{BL76}]
Given linear transformations $B=(B_j)_{j=1}^m$ where $B_j$ is a linear transformation from $\R^n$ to $\R^{n_j}$ and non-negative real numbers $p=(p_j)_{j=1}^m$, there exists a number $C\in \R$ such that for any given function tuple $f=(f_j)_{j=1}^m$ where $f_j:\R^{n_j}\rightarrow \R_+$ is a Lebesgue measurable function, the following inequality holds:
\begin{equ}[eq:BL-inequality]
\int_{x\in\R^n} \prod_{j=1}^m f_j(B_j x)^{p_j} dx \leq C \prod_{j=1}^m\left(\int_{x\in\R^{n_j}} f_j(x)dx\right)^{p_j}.
\end{equ}
The smallest $C$ such that~\eqref{eq:BL-inequality} holds for any $f$ is called the Brascamp-Lieb constant and we denote it  by $\mathrm{BL}(B,p)$. 
$(B,p)$ is called ``feasible'' if $C$ is finite.
\end{definition}
\noindent
Lieb~\cite{Lie90} showed that~\eqref{eq:BL-inequality} is saturated by Gaussian inputs: $g_j(x):=\exp(-x^\top A_j x)$ for some positive definite matrix $A_j$ for all $j$.
Evaluating~\eqref{eq:BL-inequality} with inputs $g_j$ leads to the following formulation of the Brascamp-Lieb inequality:
\begin{equ}
\left(\frac{\prod_{j=1}^m\det(A_j)^{p_j}}{\det(\sum_{j=1}^m p_j B_j^\top A_j B_j)}\right)^{1/2} \leq C.
\end{equ}
Also, Bennett {\em et al.}~\cite{bennett2008brascamp} characterized the necessary and sufficient conditions for a Brascamp-Lieb datum to be feasible.
\begin{theorem}[\bf Feasibility of Brascamp-Lieb datum \cite{bennett2008brascamp}, Theorem 1.15]
Let $(B,p)$ be a Brascamp-Lieb  datum with $B_j\in\R^{n_j\times n}$ for each $j=1,\hdots,m$. 
Then, $(B,p)$ is feasible if and only if following conditions hold:
\begin{enumerate}
	\item $n=\sum_{j=1}^m p_jn_j$, and
	\item $\dim(V)\leq \sum_{j=1}^m p_j \dim(B_j V)$ for any subspace $V$ of $\R^n$.
\end{enumerate}
\label{thm:feasibility}
\end{theorem}

\subsection{Problem statement}
Lieb's characterization leads to the following  characterization of the Brascamp-Lieb constant.

\vspace{3mm}
\fbox{
\begin{minipage}{0.8\textwidth}
{\bf Brascamp-Lieb Constant}

\vspace{1em}
{\it Input:} An $m$-tuple of matrices $B:=(B_j)_{j=1}^m$ where $B_j\in \R^{n_j\times n}$ and an $m$-dimensional vector $p$ that satisfies $n=\sum_{j=1}^m n_j p_j$.
\vspace{1em}

{\it Goal:} Compute 
\begin{align}
\mathrm{BL}(B, p) :=& \sup_{X=(X_1,\hdots,X_m)} \left(\frac{\prod_{j=1}^m\det(X_j)^{p_j}}{\det(\sum_{j=1}^m p_j B_j^\top X_j B_j)}\right)^{1/2}\label{eq:BL-optimization}\\
&\text{s.t.}\;\forall j=1,\hdots,m,\;\;X_j\in\Sn^{n_j}_{++}\nonumber
\end{align}
\end{minipage}
}

\subsection{Non-convexity of Lieb's formulation}
Let us first observe that, when $n_j=1$ for all $j$, the formulation  in ~\eqref{eq:BL-optimization} is log-concave by a change of variables similar to the example of the log of a polynomial over the positive orthant.
We can verify this observation as follows.
Let $X_1,\hdots,X_m$ be inputs with $X_j\in\Sn^{n_j}_{++}$ and let us denote $\log(p_j X_j)$ by $y_j$ for $j=1,\hdots, m$.
Then, $2\log BL(B,p; X_1,\hdots, X_m)$ becomes
\begin{equ}
f(y):=\inner{p, y} - \log(\det(\sum_{j=1}^m e^{y_j} B_j^\top B_j)) - \sum_{j=1}^m p_j\log p_j.
\end{equ}
We note that $\det(\sum_{j=1}^m e^{y_j} B_j^\top B_j)$ is a multilinear function of $(e^{y_1}, \hdots, e^{y_m})$ with non-negative coefficients.
The Cauchy-Binet formula yields
\begin{equ}
\det\left(\sum_{j=1}^m e^{y_j} B_j^\top B_j\right) = \sum_{\alpha\in \mathcal{F}} c_\alpha e^{\inner{\alpha, y}}
\end{equ}
where $\mathcal{F} := \Set{I\subseteq\{1,\hdots,m\}}{\abs{I}=n}$ and $c_\alpha=\det(B_\alpha)^2$ where $B_\alpha$ is a matrix whose rows are $B_j$ for $j\in\alpha$.
Thus, 
\begin{equ}
f(y) = \inner{p, y}-\sum_{j=1}^m p_j\log(p_j) + \log\left(\sum_{\alpha\in\mathcal{F}} c_\alpha \exp(\inner{\alpha,y})\right).
\end{equ}
A simple computation yields
\begin{align*}
\nabla^2 f(y) =& \frac{\sum_{\alpha\in\mathcal{F}} c_\alpha e^{\inner{\alpha, y}} \alpha\alpha^\top}{\sum_{\alpha\in\mathcal{F}} c_\alpha e^{\inner{\alpha, y}}} - \frac{(\sum_{\alpha\in\mathcal{F}} c_\alpha e^{\inner{\alpha, y}}\alpha)(\sum_{\alpha\in\mathcal{F}} c_\alpha e^{\inner{\alpha, y}}\alpha^\top)}{(\sum_{\alpha\in\mathcal{F}} c_\alpha e^{\inner{\alpha, y}})^2}\\
=&\frac{\sum_{\alpha,\beta\in\mathcal{F}} (\alpha-\beta)(\alpha-\beta)^\top c_\alpha c_\beta e^{\inner{\alpha+\beta, y}}}{(\sum_{\alpha\in\mathcal{F}} c_\alpha e^{\inner{\alpha,y}})}
\end{align*}
Positive semi-definiteness of $\nabla^2 f(y)$ follows from the positiveness of $\exp$ function and coefficients $c_\alpha$.

Now, let us return to general case.
In general,~\eqref{eq:BL-optimization} is not concave.
We can  get some intuition for this by noting that if we fix all but one of the $X_j$s, then $\log BL(B,p; X_1,\hdots, X_m)$  reduces to
\begin{equ}
\frac{1}{2}\left(p_j\log\det(X_j)-\log\det(p_jB_j^\top X_j B_j + C_j)\right)
\end{equ}
where $C_j$ is a positive semi definite matrix.
This function is a difference of two concave functions of $X_j$ which indicates it may not be a concave function.
In fact, for some feasible Brascamp-Lieb data one can find some input and directions such that this formulation is not concave.
We remark that although~\eqref{eq:BL-optimization} is not a log-concave function in general, it is jointly geodesically log-concave; see~\cite{vishnoi2018geodesically}.

\subsection{A geodesically convex formulation}
We present a geodesically convex formulation to compute $\mathrm{BL}(B, p)$.
Given an $m$-tuple of matrices $B:=(B_j)_{j=1}^m$ where $B_j\in \R^{n_j\times n}$ and an $m$-dimensional vector $p$ that satisfies $n=\sum_{j=1}^m n_j p_j$, define $F_{B,p}$ as follows,
\begin{equ}
F_{B,p}(X) := \sum_{j=1}^m p_j \log\det(B_j X B_j^\top) - \log\det(X).
\end{equ}
It was proved in \cite{vishnoi2018geodesically} that this is  a reformulation of the Brascamp-Lieb constant. 
\begin{theorem}[\bf Reformulation of Brascamp-Lieb constant]
If $(B,p)$ is feasible, then $BL(B, p) = \exp(-\nicefrac{1}{2}\inf_{X\in \Sn^n_{++}} F_{B,p}(X))$.
\end{theorem}

\noindent
Subsequently, it was proved that this formulation is geodesically convex.
\begin{theorem}[\bf Brascamp-Lieb constant has a succinct geodesically convex formulation\cite{vishnoi2018geodesically}]
Let $(B,p)$ be a Bracamp-Lieb datum such that $n=\sum_{j=1}^{m} p_j n_j$ and $n=\sum_{j=1}^m p_j \dim(B_j \R^n)$.
Then $F_{B,p}(X)$ is a geodesically convex function on $\Sn^n_{++}$ with respect to local metric $g_X(U,V):=\tr[X^{-1} U X^{-1} V]$.
\end{theorem}

\noindent
The conditions we impose on the Brascamp-Lieb datum are satisfied by any feasible datum, see~\cref{thm:feasibility}.
A Brascamp-Lieb satisfying these conditions is called a non-degenerate datum.

\begin{proof}
Let us first observe that $F_{B,p}$ is a positive weighted sum of 
$$-\log\det(X) \mbox{  and  } \log\det(B_j X B_j^\top)$$ for $j=1,\hdots,m$.
$-\log\det(X)$ is geodesically convex by~\cref{prop:log-det}.
Also, we know that if $T(X)$ is strictly positive linear operator, then $\log\det(T(X))$ is geodesically convex by~\cref{prop:log-det-2}.
Thus, we only need to show that $T_j(X):=B_j X B_j^\top$ is a strictly positive linear maps if $(B,p)$ is feasible.
This would imply that, for any geodesic $\gamma$ and for $t\in[0, 1]$,
\begin{align*}
	F_{B,p}(\gamma)(t)
		=& \sum_{j=1}^m p_j\log\det(B_j \gamma(t) B_j^\top) - \log\det(\gamma(t)) \\
		\leq& \sum_{j=1}^m p_j\left((1-t)\log\det(B_j \gamma(0) B_j^\top) + y \log\det(B_j \gamma(1) B_j^\top)\right)\\
		&\qquad\qquad - \left((1-t)\log\det(\gamma(0)) + t\log\det(\gamma(1))\right)\\
		=& (1-t)\left(\sum_{j=1}^m p_j \log\det(B_j \gamma(0) B_j^\top) - \log\det(\gamma(0))\right)\\
		&\qquad\qquad + t\left(\sum_{j=1}^m p_j\log\det(B_j \gamma(1) B_j^\top)-\log\det(\gamma(1))\right)\\
		=& (1-t)F_{B,p}(\gamma(0)) + t F_{B,p}(\gamma(1)).  
\end{align*}

\noindent
Let us assume that for some $j_0\in\{1,\hdots,m\}$, $T_{j_0}(X)$ is not strictly positive linear map.
Then, there exists $X_0\in\Sn^{n}_{++}$ such that $T_{j_0}(X_0)$ is not positive definite.
Thus, there exists $v\in\R^{n_{j_0}}$ such that 
$$v^\top T_{j_0}(X_0) v \leq 0.$$
Equivalently, 
$$(B_{j_0}^\top v)^\top X_0 (B_{j_0}^\top v)\leq 0.$$
Since $X_0$ is positive definite, we get $B_{j_0} v = 0$.
Hence, 
$$v^\top B_{j_0} B_{j_0}^\top v = 0.$$
Consequently, rank of $B_{j_0}$ is at most $n_{j_0}-1$ and 
$$\dim(B_{j_0} \R^n)<n_{j_0}.$$
We note that $\dim(B_j \R^n)\leq n_j$ as $B_j\in\R^{n_j\times n}$.
Therefore, 
\begin{equ}
n = \sum_{j=1}^m p_j \dim(B_j \R^n) < \sum_{j=1}^m p_j n_j = n.
\end{equ}  
This contradicts with our assumptions on $(B,p)$.
Consequently, for any $j\in\{1,\hdots,m\}$, $T_j(X):=B_j^\top X B_j$ is strictly positive linear whenever $(B,p)$ is a non-degenerate datum.
Therefore, $F_{B,p}(X)= \sum_{j\in\{1,\hdots,m\}} p_j\log\det(B_j X B_j^\top) - \log\det(X)$ is geodesically convex function.
\end{proof}

\section{Application: Operator Scaling}
\label{sec:OS}
The matrix scaling problem is a classic problem related to bipartite matchings in graphs and asks the following question: given a matrix $M\in \Z_{>0}^{n\times n}$ with positive entries, find positive vectors $x,y\in \R^{n}_{>0}$ such that
$$Y^{-1}MX ~~~~~\mbox{is a doubly-stochastic matrix},$$
where $X=\diag({x})$ and $Y=\diag({y})$.
The operator scaling problem is a generalization of this to the world of operators  as follows:

\vspace{3mm}
\fbox{
\begin{minipage}{0.8\textwidth}
{\bf Operator Scaling Problem for Square Operators}

\vspace{1em}
{\it Input:} An $m$-tuple of $n\times n$ matrices $(A_j)_{j=1}^m$
\vspace{1em}

{\it Goal:} Find two $n\times n$ matrices $X$ and $Y$ such that if $\hat{A}_i := Y^{-1} A_i X$, then
\begin{align*}
\sum_i \hat{A}_i\hat{A}_i^\top = I,\\
\sum_i \hat{A}_i^\top\hat{A}_i = I.
\end{align*}
\end{minipage}
}

\vspace{3mm}
\noindent
The Operator Scaling problem was  studied by Gurvits~\cite{gurvits2004classical} and, more recently by Garg {\em et al.}~\cite{GargGOW16}.
They showed that the Operator Scaling problem can be formulated in terms of solving the following ``operator capacity'' problem.

\vspace{3mm}
\fbox{
\begin{minipage}{0.8\textwidth}
{\bf Operator Capacity for Square Operators}

\vspace{1em}
{\it Input:} A strictly positive operator of the form $T(X):=\sum_{j=1}^m A_j X A_j^\top$ given as an  $m$-tuple of $n\times n$ matrices $(A_j)_{j=1}^m$
\vspace{1em}

{\it Goal:} Compute
\begin{equ}
\mathrm{cap}(T) := \inf_{X\in \Sn^n_{++}} \frac{\det(T(X))}{\det(X)}.
\end{equ}
\end{minipage}
}
\vspace{3mm}

\noindent
If the infimum of this problem is attainable and  is attained for $X^\star$, then it can used to compute the optimal solution to the operator scaling problem as follows:
\begin{align*}
\sum_i \hat{A}_i\hat{A}_i^\top = I,\\
\sum_i \hat{A}_i^\top\hat{A}_i = I,
\end{align*}
where $\hat{A}_i = T(X^\star)^{-1/2} A_i (X^\star)^{1/2}$.

The operator capacity turns out to be a non-convex function \cite{gurvits2004classical, GargGOW16}.
Its geodesic convexity  follows from the example  of geodesically convex functions we discussed on the positive definite cone.
In particular,~\cref{prop:log-det} asserts that $\log\det(X)$ is geodesically linear and~\cref{prop:log-det-2} asserts that $\log\det(T(X))$ is geodesically convex on the positive definite cone with respect to metric $g_X(U,V):=\tr[X^{-1}UX^{-1}V]$.
Thus, their difference $\log\mathrm{cap}(X)$ is also geodesically convex.

\begin{theorem}[AGLOW18]
Given a strictly positive linear map $T(X)=\sum_{j=1}^m A_j X A_j^\top$, $\log\mathrm{cap}(X)$ is a geodesically convex function on $\Sn^n_{++}$ with respect to local metric $g_X(U,V):=\tr[X^{-1}UX^{-1}V]$.
\end{theorem}

\newpage
\bibliographystyle{alpha}
\bibliography{refs}

\appendix
\section{Proof of the Fundamental Theorem of Riemannian Manifold}
\label{sec:levi-civita-proof}
\subsection{Uniqueness of the Levi-Civita connection}

Let $\nabla$ be a Levi-Civita connection, and $X,\,Y,\,Z$ be three vector fields in $\mathfrak{X}(M)$.  We will show that, if the assumptions of the Levi-Civita connection hold, then there is a formula which uniquely determines $\nabla$.
First, we compute the vector fields $X(g(Y,\,Z))$, $Y(g(Z,\,Y))$, and $Z(g(X,\,Y))$.  By the compatibility condition (Equation \eqref{eq:property2}), we have
\begin{align*}
X(g(Y,\,Z)) =& g(\nabla_X Y,\,Z) + g(Y,\,\nabla_X Z)\\
Y(g(Z,\,X)) =& g(\nabla_Y Z,\,X) + g(Z,\,\nabla_Y X)\\
Z(g(X,\,Y)) =& g(\nabla_Z X,\,Y) + g(X,\,\nabla_Z Y).
\end{align*}
Therefore, 
\begin{align*}
	X(g(Y,\,Z))&+Y(g(Z,\,X))-Z(g(X,\,Y))\\
	= &g(\nabla_X Y,\,Z) + g(Y,\,\nabla_X Z) + g(\nabla_Y Z,\,X) + g(Z,\,\nabla_Y X) -  g(\nabla_Z X,\,Y) - g(X,\,\nabla_Z Y) \\
	= &g(\nabla_X Y+\nabla_Y X,\,Z)+ g(\nabla_X Z - \nabla_Z X,\,Y) + g(\nabla_Y Z-\nabla_Z Y,\,X) \\
	= & g([X,\,Z],\,Y) + g([Y,\,Z],\,X) - g([X,\,Y],\,Z) + 2g(\nabla_X Y,\,Z),
\end{align*}
where the last equality is by the torsion-free property of the Levi-Civita connection.
Thus,
\begin{align}\label{eq:Koszul}
g(\nabla_X Y,\,Z) = \frac{1}{2}\bigg(
 X(g(Y,\,Z))&+Y(g(Z,\,X))-Z(g(X,\,Y))\\
 &- g([X,\,Z],\,Y) - g([Y,\,Z],\,X) + g([X,\,Y],\,Z)   \bigg)\nonumber
\end{align}
for all $X,\,Y,\,Z\in\mathfrak{X}(M)$.
Since $g$ is non-degenerate, $\nabla_X Y$ is uniquely determined by Equation \eqref{eq:Koszul}, implying that the Levi-Civita connection is unique.
The Christoffel symbols of this unique connection are computed explicitly in the next section.

\subsection{Formula for Christoffel symbols in terms of the metric}
Equation \eqref{eq:Koszul}  is called the Koszul formula.  Using the Koszul formula we can compute Christoffel symbols for the Levi-Civita connection.
Recall that $\Gamma_{ij}^k$ denotes the coefficient of $(\nabla_{\partial_i} \partial_j)$ in the direction of $\partial_k$.
In other words,
\begin{equ}
\nabla_{\partial_i} \partial_j = \sum_{k=1}^d \Gamma_{ij}^k \partial_k,
\end{equ}
where the convention is that we sum over all indices that only appear on one side of the equation (in this case, we sum over the index $k$).
Define the matrix $G:= [g_{ij}]_{n\times n}$ where each entry is $g_{ij}:=g(\partial_i,\,\partial_j)$.  Denote each entry $(G)^{-1}_{ij}$ of its inverse by $g^{ij}$.
We have
\begin{align*}
g(\nabla_{\partial_i}\partial_j,\,\partial_l) 
	=& \sum_{k=1}^d g(\Gamma_{ij}^k \partial_k,\,\partial_l)\\
	=& \sum_{k=1}^d  \Gamma_{ij}^k g_{kl}\\
	=& \sum_{k=1}^d  g_{lk} \Gamma_{ij}^k.
\end{align*}
In matrix form,
\begin{equ}
G \begin{bmatrix}
\Gamma_{ij}^1\\
\vdots\\
\Gamma_{ij}^n
\end{bmatrix}
=
\begin{bmatrix}
g(\nabla_{\partial_i}\partial_j,\,\partial_1)\\
\vdots\\
g(\nabla_{\partial_i}\partial_j,\,\partial_n)
\end{bmatrix}.
\end{equ}
Thus, $\Gamma_{ij}^k$ is given by
\begin{equ}\label{eq:0}
\Gamma_{ij}^k= \sum_{l=1}^d g^{kl} g(\nabla_{\partial_i}\partial_j,\,\partial_l).
\end{equ}
We can compute $g(\nabla_{\partial_i}\partial_j,\,\partial_l)$ by setting $X$ to be $\partial_i$, $Y$ to be $\partial_j$, and $Z$ to be $\partial_l$ in Equation \eqref{eq:Koszul}.
Before computing $g(\nabla_{\partial_i}\partial_j,\,\partial_l)$, 
we need a consequence of torsion-freeness $\Gamma_{ij}^k = \Gamma_{ji}^k$ for any $i,j,k$.
Let us recall that torsion-freeness requires for any vector fields $X$ and $Y$,
\begin{equ}
\nabla_X Y - \nabla_Y X = [X, Y].
\end{equ}
In particular, if $X=\partial_i$ and $Y=\partial_j$, then
\begin{align*}
\nabla_{\partial_i}\partial_j - \nabla_{\partial_j} \partial_i 
	=& \sum_{k=1}^d \Gamma_{ij}^k \partial_k - \sum_{k=1}^d \Gamma_{ji}^k\partial_k = \sum_{k=1}^d(\Gamma_{ij}^k-\Gamma_{ji}^k)\partial_k\\
\;[\partial_i,\partial_j]
	=& \partial_i\partial_j-\partial_j\partial_i = 0
\end{align*}
$\partial_k$s are linearly independent.
Consequently, their coefficients $\Gamma_{ij}^k-\Gamma_{ji}^k$ should be 0.
Therefore,
\begin{equ}
g(\nabla_{\partial_i}\partial_j,\,\partial_l)
	= \sum_{k=1}^d  g_{lk} \Gamma_{ij}^k
	= \sum_{k=1}^d  g_{lk} \Gamma_{ji}^k
	= \sum_{k=1}^d  g(\nabla_{\partial_j}\partial_i,\,\partial_l).
\end{equ}
Therefore, Equation \eqref{eq:Koszul} implies that
\begin{align}\label{eq:1}
g(\nabla_{\partial_i}\partial_j,\,\partial_l) =& \frac{1}{2}\left(
 \partial_i(g(\partial_j,\,\partial_l))
 +\partial_j(g(\partial_l,\,\partial_i))
 -\partial_l(g(\partial_i,\,\partial_j))\right.\\
 &\qquad\qquad
\left. -g([\partial_i,\,\partial_l],\,\partial_j)
 - g([\partial_j,\,\partial_l],\,\partial_i) 
 + g([\partial_i,\,\partial_j],\,\partial_l\right),\nonumber
\end{align}
and
\begin{align}
g(\nabla_{\partial_i}\partial_j,\,\partial_l) &=  g(\nabla_{\partial_j}\partial_i,\,\partial_l)\nonumber\\
&=  \frac{1}{2}\left(
 \partial_j(g(\partial_i,\,\partial_l))
 +\partial_i(g(\partial_l,\,\partial_j))
 -\partial_l(g(\partial_j,\,\partial_i))\right.\label{eq:1'}\\
 &\qquad\qquad
\left. -g([\partial_j,\,\partial_l],\,\partial_i) 
 - g([\partial_i,\,\partial_l],\,\partial_j) 
 + g([\partial_j,\,\partial_ i],\,\partial_l) 
\right).\nonumber
\end{align}
Since $g$ is symmetric, we have
\begin{align*}
 \partial_j(g(\partial_l,\,\partial_i))=& \partial_j(g(\partial_i,\,\partial_l)),\\
 \partial_i(g(\partial_j,\,\partial_l))=& \partial_i(g(\partial_l,\,\partial_j)),\\
 \partial_l(g(\partial_i,\,\partial_j))=& \partial_l(g(\partial_j,\,\partial_i)).
 \end{align*}
Thus, combining Equations \eqref{eq:1} and \eqref{eq:1'}, we have
\begin{equ}
g([\partial_i,\,\partial_j],\,\partial_l) = g([\partial_j,\,\partial_i],\,\partial_l).
\end{equ}
Since $\nabla$ is torsion-free, we have that 
$$[\partial_i,\,\partial_j] = -[\partial_j,\,\partial_i],$$
 implying that 
 $$g([\partial_i,\,\partial_j],\,\partial_l)= 0.$$
Since our selection of indices $i,\,j,\,l$ was arbitrary,
Equation \eqref{eq:1} simplifies to
\begin{equ}[eq:2]
g(\nabla_{\partial_i}\partial_j,\,\partial_l)  = \frac{1}{2}(\partial_i g_{jl}+\partial_{j} g_{il} - \partial_l g_{ij}).
\end{equ}
Combining \eqref{eq:0} and \eqref{eq:2}, we get
\begin{equ}
\Gamma_{ij}^k = \sum_{l=1}^d \frac{1}{2}g^{kl}(\partial_i g_{jl}+\partial_{j} g_{il} - \partial_l g_{ij}).
\end{equ}

\section{Euler-Lagrange Dynamics on a Riemannian Manifold}
\label{sec:EL}
In this section, we derive the Euler-Lagrange dynamics, a useful concept from the calculus of variations.
The significance of Euler-Lagrange equations is that it allows us to characterize minimizers of following work or action integral:
\begin{equ}
S[\gamma] = \int_a^b L(\gamma, \dot{\gamma}, t) dt
\end{equ}
where $L$ is a smooth function, $\gamma$ is a curve whose end-points fixed, and $\dot{\gamma}=\frac{d \gamma}{dt}$.
An application of this characterization is~\cref{thm:geodesic-EL}.
Before we delve into derivations of Euler-Lagrange equations on a Riemannian manifold, we present a point of view to Euler-Lagrange equations from a Newtonian mechanics perspective.

\subsection{Euler-Lagrange dynamics in Newtonian mechanics} 
In Euler-Lagrangian dynamics, one typically denotes the  generalized position of a particle by $q$ and its generalized velocity by $\dot{q}$, as it is normally a time derivative of the position. In most cases we should still think of $\dot{q}$ as a formal variable independent of $q$, but if we consider a trajectory $(q(t),\dot{q}(t))$ of a particle, then clearly
$$\frac{d}{dt} q(t) = \dot{q}(t).$$
\noindent 
The central object of the Euler-Lagrangian dynamics is the {\it Lagrangian} function defined as
\begin{equ}
L(q_1,\,\hdots,\,q_d,\,\dot{q}_1,\,\hdots,\,\dot{q}_d,\,t) = K(q_1,\,\hdots,\,q_d,\,\dot{q}_1,\,\hdots,\,\dot{q}_d,\,t) -V(q_1,\,\hdots,\,q_d,\,t),
\end{equ}
where $K$ is to be thought of as kinetic energy  and $V$ as  potential energy.
An example of Lagrangian equation in 1-dimension is 
\begin{equ}\label{eq:ex_L}
L = \frac{1}{2} m\dot{q}^2 - V(q).
\end{equ}
The Lagrangian can also depend on time, for instance if $\beta<0$ is a fixed constant
\begin{equ}
L' = e^{\beta t} \left(\frac{1}{2} m\dot{q}^2 - V(q)\right),
\end{equ}
then the above describes a physical system with friction.

We are now ready to state the Euler-Lagrange equations.
\begin{definition}\
The Euler-Lagrange equations describing the motion of a particle $(q,\dot{q})$ with respect to a Lagrangian $L(q,\dot{q},t)$ are given by
\begin{equation}\label{eq:el}
~~~~~~\frac{\partial L}{\partial q_i} = \frac{d}{dt}\frac{\partial L}{\partial \dot{q}_i}~~~~~~~i=1,2, \ldots, d.
\end{equation}
\end{definition}
\noindent
We can see for instance that the Euler-Lagrange equations with respect to~\eqref{eq:ex_L} are compatible with Newtonian dynamics

$$\frac{d}{dt}\frac{\partial L}{\partial \dot{q}}= \frac{d}{dt}(m\dot{q})
		=  \frac{d}{dt} p
		= F(q)
		=  -\nabla V(q)
	=\frac{\partial L}{\partial q},
$$

\noindent 
where $q$ is the position of particle, $m$ is the mass of the particle, $p$ is the momentum of the particle, $F(q)$ is the force applied at position $q$, and $V(q)$ is the potential energy at position $p$.
The quantity $\frac{\partial L}{\partial q_i}$ is usually referred to as conjugate momenta or generalized momenta, while $\frac{\partial L}{\partial \dot{q}_i}$ is called the generalized force.
We note that the quantity
$$S(q):=\int_a^b L(q,\dot{q},t)dt,$$
represents the work performed by the particle.
The {\it Principle of Least Action} asserts that the trajectory taken by the particle to go from state $(q(a),\dot{q}(a))$ to state $(q(b), \dot{q}(b))$ will minimize $S(q)$. 
In other words, if we compute partial derivative of $S$ with respect to the trajectory of the particle, then result should be $0$.
Starting from this observation and using properties of differentiation and integrals one can arrive to equations~\eqref{eq:el}.
Here, we present the proof of Euler-Lagrange in a more general setting of an ``embedded manifold''.

\subsection{Derivation of Euler-Lagrange equations on a Riemannian manifold}
In this section, we derive Euler-Lagrange equations on a manifold $M$ that is embedded in a larger Euclidean space.
In other words, we assume that $M$ and $TM$ are both part of $\R^n$.
This assumption allows us to use usual the calculus rules and is sufficient for all the applications discussed in this exposition.
One can extend the ideas presented in this section to general manifold setting by carefully handling local coordinate changes.

\begin{theorem}[\bf Euler-Lagrange equations]
Let $(M,g)$ be a Riemannian manifold and $p,q\in M$.
Consider the set of all smooth curves $\Gamma$ that join $p$ to $q$ by mapping $[0,1]$ to $M$.
Let us consider a smooth real-valued function $L(\gamma, \dot{\gamma},t)$ for $\gamma\in\Gamma$ and
\begin{equ}[eq:EL-S-def]
S[\gamma] := \int_0^1 L(\gamma, \dot{\gamma}, t)dt.
\end{equ}
If $\gamma'\in \Gamma$ is a minimizer of $S[\gamma]$, then
\begin{equ}
\frac{d}{dt} \frac{\partial L}{\partial \dot{\gamma}}(\gamma') - \frac{\partial L}{\partial \gamma}(\gamma') = 0.
\end{equ}
\label{thm:EL}
\end{theorem}

\noindent
In order to prove~\cref{thm:EL}, we need to argue about the change induced in $S$ by changing $\gamma$.
Towards this end, we introduce the \emph{variation} of a curve.

\begin{definition}[\bf Variation of a curve]
Let $M$ be a differentiable manifold and $\gamma:[0,1]\to M$ be a smooth curve.
A variation of $\gamma$ is a smooth function $\nu:(-c,c)\times [0,1]\to M$ for some $c>0$ such that $\nu$ is $\nu(0, t) = \gamma(t)$, $\forall t\in[0,1]$ and $\nu(u,0)=\gamma(0)$, $\nu(u,1)=\gamma(1)$, $\forall u\in(-c,c)$.
\label{def:variation}
\end{definition}

\begin{figure}[!hbt]
\centering
\includegraphics[keepaspectratio, width=0.45\textwidth]{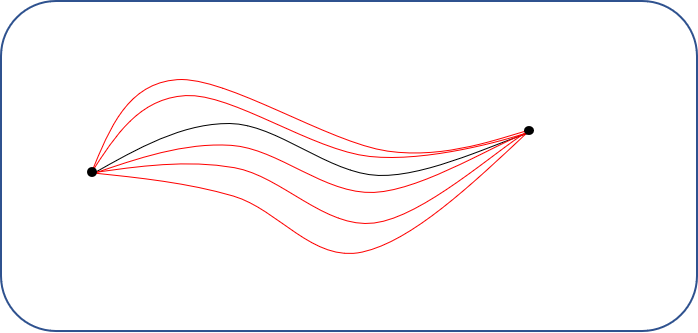}
\caption{Variation of a curve. The black curve is the original curve and green curves are its smooth variations.}
\end{figure}

\begin{proof}[\bf Proof of~\cref{thm:EL}]
$\gamma$ is a minimizer of $S$ if any small variation on $\gamma$  causes an increase in $S[\gamma]$.
More formally, if $\nu:(-c,c)\times[0,1]\to M$ is a variation of $\gamma$, then 
\begin{equ}
S[\nu_0] \leq S[\nu_u]
\end{equ}
for any $u\in(-c,c)$ where $\nu_u(t):=\nu(u,t)$.
Consequently,
\begin{equ}
\left.\frac{d S[\nu_u]}{d u}\right\rvert_{u=0} = 0.
\end{equ}
Now, let us evaluate derivative of $S[\nu_u]$ using~\eqref{eq:EL-S-def}.
\begin{align*}
0 	=& \left.\frac{d S[\nu_u]}{d u}\right\rvert_{u=0}\\
	=& \left.\frac{d}{du}\int_0^1 L\left(\nu_u, \frac{\partial \nu_u}{\partial t}, t\right)dt\right\rvert_{u=0}\\
	=& \int_0^1 \left.\frac{d}{du}  L\left(\nu_u, \frac{\partial \nu_u}{\partial t}, t\right)\right\rvert_{u=0} dt\\
	=& \int_0^1 \left.\inner*{\frac{\partial L\left(\nu_u, \frac{\partial \nu_u}{\partial t}, t\right)}{\partial \nu_u},\frac{\partial \nu_u}{\partial u}}\right\rvert_{u=0}
		+\left.\inner*{\frac{\partial L\left(\nu_u, \frac{\partial \nu_u}{\partial t}, t\right)}{\partial \left(\nicefrac{\partial \nu_u}{\partial t}\right)},\frac{\partial \left(\nicefrac{\partial \nu_u}{\partial t}\right)}{\partial u}}\right\rvert_{u=0}dt.
\end{align*}
We note that 
\begin{equ}
\frac{\partial \left(\nicefrac{\partial \nu_u}{\partial t}\right)}{\partial u} = \frac{\partial^2 \nu(u, t)}{\partial u \partial t} = \frac{\partial^2 \nu(u, t)}{\partial t \partial u} = \frac{\partial \left(\nicefrac{\partial \nu_u}{\partial u}\right)}{\partial t},
\end{equ}
by smoothness.
Let us denote $\left.\nicefrac{\partial \nu_u}{\partial u}\right\rvert_{u=0}$ by $\mu$.
Then by the chain rule,
\begin{align*}
0
	=& \int_0^1 \left.\inner*{\frac{\partial L\left(\nu_u, \frac{\partial \nu_u}{\partial t}, t\right)}{\partial \nu_u},\frac{\partial \nu_u}{\partial u}}\right\rvert_{u=0}
		+\left.\inner*{\frac{\partial L\left(\nu_u, \frac{\partial \nu_u}{\partial t}, t\right)}{\partial \left(\nicefrac{\partial \nu_u}{\partial t}\right)},\frac{\partial \left(\nicefrac{\partial \nu_u}{\partial t}\right)}{\partial u}}\right\rvert_{u=0}dt\\
	=& \int_0^1 \inner*{
		\left.\frac{\partial L\left(\nu_u, \frac{\partial \nu_u}{\partial t}, t\right)}{\partial \nu_u}\right\rvert_{\nu_u=\gamma},
		\left.\frac{\partial \nu_u}{\partial u}\right\rvert_{u=0}}
		+\inner*{
		\left.\frac{\partial L\left(\nu_u, \frac{\partial \nu_u}{\partial t}, t\right)}{\partial \left(\nicefrac{\partial \nu_u}{\partial t}\right)}\right\rvert_{\nicefrac{\partial \nu_u}{\partial t}=\dot{\gamma}},
		\left.\frac{\partial \left(\nicefrac{\partial \nu_u}{\partial t}\right)}{\partial u}\right\rvert_{u=0}}dt\\
	=& \int_0^1 \inner*{\frac{\partial L(\gamma,\dot{\gamma}, t)}{\partial \gamma}, \mu} + \inner*{\frac{\partial L(\gamma, \dot{\gamma}, t)}{\partial \dot{\gamma}}, \frac{d \mu}{dt}}dt.
\end{align*}
We note that $\mu(0)=\mu(1)=0$ as $\nu_u(0)=\gamma(0)$ and $\nu_u(1)=\gamma(1)$ for $u\in(-c,c)$.
Thus, integration by parts, we obtain
\begin{align*}
\int_0^1  \inner*{\frac{\partial L(\gamma, \dot{\gamma}, t)}{\partial \dot{\gamma}}, \frac{d \mu}{dt}}dt 
	=& \inner*{\frac{\partial L(\gamma, \dot{\gamma}, t)}{\partial \dot{\gamma}}, \left.\mu\right\rvert_{0}^1}
	 -\int_0^1  \inner*{\frac{d}{dt}\frac{\partial L(\gamma, \dot{\gamma}, t)}{\partial \dot{\gamma}}, \mu}dt \\
	 =& -\int_0^1  \inner*{\frac{d}{dt}\frac{\partial L(\gamma, \dot{\gamma}, t)}{\partial \dot{\gamma}}, \mu}dt .
\end{align*}
Consequently, we have
\begin{equ}[eq:EL-proof-0]
0 = \int_0^1 \inner*{\frac{\partial L(\gamma, \dot{\gamma}, t)}{\partial \gamma}-\frac{d}{dt}\frac{\partial L(\gamma, \dot{\gamma}, t)}{\partial \dot{\gamma}}, \mu}dt.
\end{equ}
If $\gamma$ is a minimizer of $S$, then~\eqref{eq:EL-proof-0} should hold for any variation.
As $\mu=\left.\nicefrac{\partial \nu_u}{\partial u}\right\rvert_{u=0}$, we can pick $\mu$ to be any smooth vector field along $\gamma$.
In particular, we can pick $\mu$ as
\begin{equ}
\mu =\frac{\partial L(\gamma, \dot{\gamma}, t)}{\partial \gamma} - \frac{d}{dt}\frac{\partial L(\gamma, \dot{\gamma}, t)}{\partial \dot{\gamma}},
\end{equ}
since $\gamma$ is a smooth function.
Consequently,
\begin{align*}
0=& \int_0^1 \inner*{\frac{\partial L(\gamma, \dot{\gamma}, t)}{\partial \gamma} - \frac{d}{dt}\frac{\partial L(\gamma, \dot{\gamma}, t)}{\partial \dot{\gamma}}, \mu}dt \\
=& \int_0^1\inner*{\frac{\partial L(\gamma, \dot{\gamma}, t)}{\partial \gamma} - \frac{d}{dt}\frac{\partial L(\gamma, \dot{\gamma}, t)}{\partial \dot{\gamma}}, \frac{\partial L(\gamma, \dot{\gamma}, t)}{\partial \gamma} - \frac{d}{dt}\frac{\partial L(\gamma, \dot{\gamma}, t)}{\partial \dot{\gamma}}}dt.
\end{align*}
This implies that
\begin{equ}
\inner*{\frac{\partial L(\gamma, \dot{\gamma}, t)}{\partial \gamma} - \frac{d}{dt}\frac{\partial L(\gamma, \dot{\gamma}, t)}{\partial \dot{\gamma}}, \frac{\partial L(\gamma, \dot{\gamma}, t)}{\partial \gamma} - \frac{d}{dt}\frac{\partial L(\gamma, \dot{\gamma}, t)}{\partial \dot{\gamma}}} = 0
\end{equ}
as it is non-negative everywhere.
Finally, 0 vector is the only vector whose norm is 0.
Therefore,
\begin{equ}
\frac{\partial L(\gamma, \dot{\gamma}, t)}{\partial \gamma} - \frac{d}{dt}\frac{\partial L(\gamma, \dot{\gamma}, t)}{\partial \dot{\gamma}} = 0.
\end{equ}
\end{proof}

\section{Minimizers of Energy Functional are Geodesics}
\label{sec:geodesic-equivalence}
In this section, we prove~\cref{thm:geodesic-equivalence}.
Given a curve $\gamma$ that minimize the energy functional, we derive the differential equations describing $\gamma$ in terms of Christoffel symbols and metric.
Later, we verify that these differential equations implies that $\nabla_{\dot{\gamma}}\dot{\gamma}=0$ where $\nabla$ is the Levi-Civita connection.
In other words, $\gamma$ is a geodesic with respect to the Levi-Civita connection.

\begin{proof}[Proof of~\cref{thm:geodesic-equivalence}]
Let us assume that $M$ is a $d$-dimensional manifold.
Let us fix a frame bundle, $\{\partial_i\}_{i=1}^d$, for $M$ and $\Gamma_{ij}^k$ be the Christoffel symbols corresponding to Levi-Civita connection.
Let us also define
\begin{equ}
\mathcal{E}(t):=\frac{1}{2}\sum_{i,j=1}^d g_{ij}(\gamma(t)) \dot{\gamma}_i(t)\dot{\gamma}_j(t).
\end{equ} 
Before we use~\cref{thm:geodesic-EL}, let us compute partial derivatives of $\mathcal{E}$.
\begin{equs}
\frac{\partial \mathcal{E}}{\partial \gamma_i} 
	=& \frac{1}{2}\sum_{j,k=1}^d \frac{\partial g_{jk}(\gamma(t))\dot{\gamma}_j(t)\dot{\gamma}_k(t)}{\partial \gamma_i}\\
	=& \frac{1}{2}\sum_{j,k=1}^d \frac{\partial g_{jk}(\gamma(t))}{\partial \gamma_i}\dot{\gamma}_j(t)\dot{\gamma}_k(t)\\
	=& \frac{1}{2}\sum_{j,k=1}^d \partial_i(g_{jk}(\gamma(t)))\dot{\gamma}_j(t)\dot{\gamma}_k(t)\\
\frac{\partial \mathcal{E}}{\partial \dot{\gamma}_i} 
	=& \frac{1}{2}\sum_{j,k=1}^d \frac{\partial g_{jk}(\gamma(t))\dot{\gamma}_j(t)\dot{\gamma}_k(t)}{\partial \dot{\gamma}_i}\\
	=& \sum_{j=1}^d  g_{ij}(\gamma(t))\dot{\gamma}_j(t)\\
\frac{d}{dt}\frac{\partial \mathcal{E}}{\partial \dot{\gamma}_i} 
	=& \frac{d}{dt} \sum_{j=1}^d  g_{ij}(\gamma(t))\dot{\gamma}_j(t)\\
	=& \sum_{j=1}^d \frac{d g_{ij}(\gamma(t))}{dt} \dot{\gamma}_j(t) + g_{ij}(\gamma(t))\frac{d \dot{\gamma}_j(t)}{dt}\\
	=& \sum_{j=1}^d \dot{\gamma}(t)(g_{ij}(\gamma(t))) \dot{\gamma}_j(t) + g_{ij}(\gamma(t))\ddot{\gamma}_j(t)\\
	=& \sum_{j,k=1}^d \partial_k(g_{ij})(\gamma(t)) \dot{\gamma}_k(t)\dot{\gamma}_j(t) + \sum_{j=1}^d g_{ij}(\gamma(t))\ddot{\gamma}_j(t)
\end{equs}
for $i=1,\hdots,d$.
Thus,~\cref{thm:geodesic-EL} implies that
 implies that
\begin{align*}
0
	=& \frac{d}{dt} \left(\frac{\partial \mathcal{E}}{\partial \dot{\gamma}_i}\right) - \frac{\partial \mathcal{E}}{\partial \gamma_i}\\
	=& \sum_{j,k=1}^d \partial_k(g_{ij}(\gamma(t))) \dot{\gamma}_k(t)\dot{\gamma}_j(t) + g_{ij}(\gamma(t))\ddot{\gamma}_j(t)-\frac{1}{2}\sum_{j,k=1}^d \partial_i(g_{jk}(\gamma(t)))\dot{\gamma}_j(t)\dot{\gamma}_k(t)\\
	=& \frac{1}{2}\sum_{j,k=1}^d(\partial_k g_{ij}(\gamma(t)) + \partial_j g_{ik}(\gamma(t))-\partial_i g_{jk}(\gamma(t)))\dot{\gamma}_j(t)\dot{\gamma}_k(t)+\sum_{j=1}^dg_{ij}(\gamma(t))\ddot{\gamma}_j(t)
\end{align*}
for $i=1,\hdots,d$.
Consequently,
\begin{align*}
\ddot{\gamma}_l(t)
	=&\sum_{j=1}^d \delta_{lj} \ddot{\gamma}_j(t)\\
	=&\sum_{i,j=1}^d g^{li}(\gamma(t))g_{ij}(\gamma(t))\ddot{\gamma}_j(t) \\
	=&- \frac{1}{2}\sum_{j,k=1}^d\sum_{i=1}^dg^{li}(\gamma(t))(\partial_k g_{ij}(\gamma(t)) + \partial_j g_{ik}(\gamma(t))-\partial_i g_{jk}(\gamma(t)))\dot{\gamma}_j(t)\dot{\gamma}_k(t)\\
	=& -\sum_{j,k=1}^d\Gamma^l_{jk}\dot{\gamma}_j(t)\dot{\gamma}_k(t)
\end{align*}
for $l=1,\hdots,d$ where $\delta_{lj}$ is Kronecker delta.
The second equality follows from $\sum_{i=1}^dg^{li}(\gamma(t))g_{ij}(\gamma(t))$ is equal to $lj$th entry of matrix multiplication $(G^{-1} G)$ which is 1 if $l=j$ and 0 otherwise.
Therefore,
\begin{equ}
\sum_{l=1}^d\left(\sum_{j,k=1}^d \dot{\gamma}_j(t)\dot{\gamma}_k(t)\Gamma_{jk}^l + \ddot{\gamma}_l\right)\partial_l = 0.
\end{equ}
Hence, $\nabla_{\dot{\gamma}}\dot{\gamma}=0$ by~\cref{prop:geodesic-christoffel}.
\end{proof}

\end{document}